\documentclass[sn-basic]{sn-jnl}

\jyear{2022}%

\usepackage{amsmath,amsthm,amssymb,amsfonts,amscd,graphicx}
\usepackage{comment}
\usepackage{natbib}

\renewcommand{\exp}[1]{\operatorname{exp}\left(#1\right)} 
\providecommand{\argmin}{\mathop\mathrm{arg min}}

\newcommand{\eps}{\varepsilon}
\newcommand{\diam}{\mathrm{diam}}
\newcommand{\wh}{\widehat}
\newcommand{\wt}{\widetilde}

\newcommand{\rom}[1]{\uppercase\expandafter{\romannumeral #1\relax}}
\newcommand{\be}{\begin{equation*}}
\newcommand{\ee}{\end{equation*}}
\newcommand{\ben}{\begin{equation}}
\newcommand{\een}{\end{equation}}
\newcommand{\bmln}{\begin{multline*}}
\newcommand{\emln}{\end{multline*}}

\def\l{\left}
\def\r{\right}
\newcommand{\m}{\mathcal}
\newcommand{\mb}{\mathbb}

\newcommand{\hL}{\widehat{L}}

\newcommand{\pr}[1]{\mathbb{P}{\left(#1\right)}}

\newcommand{\pd}{\partial}

\def\ml#1{\begin{multline*}{#1}\end{multline*}}
\def\mln#1{\begin{multline}{#1}\end{multline}}

\newtheorem*{thm}{Theorem}

\newcommand{\scolor}[1]{\textcolor{cyan}{#1}} 
\newcommand{\shunancolor}[1]{\textcolor{red}{#1}}

\newcommand{\yrcite}[1]{\cite{#1}}

\theoremstyle{thmstyleone}%
\newtheorem{theorem}{Theorem}
\newtheorem{lemma}{Lemma}

\theoremstyle{thmstyletwo}%
\newtheorem{example}{Example}%
\newtheorem{remark}{Remark}%
\newtheorem{assumption}{Assumption}

\theoremstyle{thmstylethree}%

\begin{document}

\title[Robust Bayesian inference]{Generalized Median of Means Principle for Bayesian Inference}


\author[1]{\fnm{Stanislav} \sur{Minsker}}\email{minsker@usc.edu}

\author*[2]{\fnm{Shunan} \sur{Yao}}\email{yaoshunan@hkbu.edu.hk}

\affil[1]{\orgdiv{Department of Mathematics}, \orgname{University of Southern California}, \orgaddress{\street{3620 S. Vermont Ave}, \city{Los Angeles}, \postcode{90007}, \state{California}, \country{USA}}}

\affil[2]{\orgdiv{Department of Mathematics}, \orgname{Hong Kong Baptist University}, \orgaddress{\street{224 Waterloo Road}, \city{Kowloon Tong}, \state{Kowloon}, \country{Hong Kong}}}


\abstract{The topic of robustness is experiencing a resurgence of interest in the statistical and machine learning communities. In particular, robust algorithms making use of the so-called median of means estimator were shown to satisfy strong performance guarantees for many problems, including estimation of the mean, covariance structure as well as linear regression. In this work, we propose an extension of the median of means principle to the Bayesian framework, leading to the notion of the robust posterior distribution. In particular, we (a) quantify robustness of this posterior to outliers, (b) show that it satisfies a version of the Bernstein-von Mises theorem that connects Bayesian credible sets to the traditional confidence intervals, and (c) demonstrate that our approach performs well in applications. }




\maketitle

\section{Introduction.}
\label{sec:intro}

Modern statistical and machine learning algorithms typically operate under limited human supervision, therefore robustness - the ability of algorithms to properly handle atypical or corrupted inputs - is an important and desirable property. Robustness of the most basic algorithms, such as estimation of the mean and covariance structure that serve as ``building blocks'' of more complex methods, have received significant attention in the mathematical statistics and theoretical computer science communities; the survey papers by \citet{lugosi2019mean,diakonikolas2019recent} provide excellent overview of the recent contributions of these topics as well as applications to a variety of statistical problems. 
The key defining characteristics of modern robust methods are (a) their ability to operate under minimal model assumptions; (b) ability to handle high-dimensional inputs and (c) computational tractability. 
However, many algorithms that provably admit strong theoretical guarantees are not computationally efficient. In this work, we rely on a class of methods that can be broadly viewed as \emph{risk minimization}: the output (or the solution) provided by such methods is always a minimizer of the properly defined risk, or cost function. For example, estimation of the mean $\mu$ of a square-integrable random variable $Z$ can be viewed as minimization of the risk $L(\theta) = \mb E(Z - \theta)^2$ over $\theta\in \mb R$. Since the risk involves the expectation with respect to the unknown distribution, its exact computation is impossible. 
Instead, risk minimization methods introduce a robust data-dependent ``proxy'' of the risk function, and attempt to minimize it instead. The robust empirical risk minimization method by  \cite{brownlees2015empirical}, the ``median of means tournaments'' developed by \citet{lugosi2016risk} and a closely related method due to \yrcite{lecue2020robust2} are the prominent examples of this approach. Unfortunately, the resulting problems are computationally hard as they typically involve minimization of general non-convex functions. In this paper, we propose a Bayesian analogue of robust empirical risk minimization that allows one to replace non-convex loss minimization by sampling that can be readily handled by many existing MCMC algorithms. Moreover, we show that for the parametric models, our approach preserves one of the key benefits of Bayesian methods - the ``built-in'' quantification of uncertainty - and leads to the asymptotically valid confidence sets when the models are well-specified. At the core of our method is a version of the median of means principle, and our results demonstrate its potential beyond the widely studied applications in the statistical learning framework.

Next, we introduce the mathematical framework used throughout the text. Let $\tilde X$ be a random variable with values in some measurable space and unknown distribution $P$. Suppose that $\tilde{\m X}_N:=\l(\tilde X_1,\ldots,\tilde X_N \r)$ are the training data -- N i.i.d. copies of 
$\tilde X$. We assume that the sample has been modified in the following way: an ``adversary'' replaces a random set of $\m O<N$ observations by arbitrary values, possibly depending on the sample. Only the corrupted values $\m X_N:=\l(X_1,\ldots,X_N \r)$ are observed.

Suppose that $P$ has a density $p$ with respect to a $\sigma$-finite measure $\mu$ (for instance, the Lebesgue measure or the counting measure). We consider a family of density functions $\{p_\theta (\cdot), \theta \in \Theta\}$, where $\Theta \subset \mb R^d$ is a compact subset. However, $p$ is not necessarily in $\{p_\theta (\cdot), \theta \in \Theta\}$. We assume that the Kullback-Leibler (KL) divergence of any $p_\theta$ from $p$, i.e., $\mb E \log \frac{p(X)}{p_\theta (X)}$ is finite. We also make the necessary identifiability assumption that there exists an unknown and unique $\theta_0$ in the interior of $\Theta$ that minimizes the KL divergence from $p$ in $\{p_\theta (\cdot), \theta \in \Theta\}$. That is, $\theta_0 = \argmin_{\theta \in \Theta} \mb E \log \frac{p(X)}{p_\theta (X)}$. If $p \in \{p_\theta (\cdot), \theta \in \Theta\}$, $p \equiv p_{\theta_0}$. Otherwise, $p_{\theta_0}$ is the closest point in $\Theta$ to $p$ in terms of KL divergence. Equivalently, $\theta_0$ is the unique minimizer \footnote{Here, $\theta'\in \Theta$ is fixed but arbitrary, and we suppress it in the notation for brevity.} of $L(\theta) := \mb E(\ell(\theta, X) - \ell(\theta',X))$
over $\theta\in \Theta$, where $\ell(\theta, \cdot)$ is the negative log-likelihood, that is, $\ell(\theta, \cdot) = - \log p_\theta (\cdot)$. Clearly, an approach based on minimizing the classical empirical risk $L_N(\theta):=\frac{1}{N}\sum_{j=1}^N (\ell(\theta,X_j) - \ell(\theta',X_j))$ of $\mb E\ell(\theta, X)$ leads to the familiar maximum likelihood estimator (MLE) $\theta^\ast_N$. 
At the same time, the main object of interest in the Bayesian approach is the \emph{posterior distribution}, which is a random probability measure on $\Theta$ defined via
\begin{equation}
\label{eq:posterior}
\Pi_{N}(B\mid X_N)= \frac{\int\limits_B \prod_{j=1}^N p_\theta(X_j)d\Pi(\theta)}
{\int\limits_{\Theta}\prod_{j=1}^N p_\theta(X_j) d\Pi(\theta)}
\end{equation}
for all measurable sets $B\subseteq \Theta$. Here, $\Pi$ is the \emph{prior} distribution with density $\pi(\cdot)$ with respect to the Lebesgue measure. The following result, known as the Bernstein-von Mises (BvM) theorem that is due to L. Le Cam in its present form (see the book by \cite{van2000asymptotic} for its proof and discussion), is one of the main bridges connecting the frequentist and Bayesian approaches.
\begin{thm}[Bernstein-von Mises]
Assume that $p = p_{\theta_0}$ for some $\theta_0\in \Theta$ (that is, the model is well-specified). Then,  under the appropriate regularity assumptions on the family $\{p_\theta, \ \theta\in \Theta\}$,
\begin{align*}
    \left\|{\Pi}_N - \mathcal{N} \left(\theta^\ast_{N}, \frac{1}{N}\l(I(\theta_0)\r)^{-1}\right)\right\|_{\mathrm{TV}} \overset{P}{\longrightarrow} 0.
\end{align*} 
\end{thm}
Here, $\theta^\ast_N$ is the MLE, $\|\cdot\|_{\mathrm{TV}}$ stands for the total variation distance, $I(\theta_0)$ is the Fisher Information matrix $I(\theta) = \mb E \l[\l( \partial_\theta \ell(\theta,X)\r)\l(\partial_\theta \ell(\theta,X)\r)^T\r]$ and $\overset{P}{\longrightarrow}$ denotes convergence in probability (with respect to the distribution of the sample $\tilde {\m X}_N$). A more general version of this result that holds for possibly misspecified models has been established by \citet{kleijn2012bernstein}.

In essence, BvM theorem implies that when the model is well specified, the $1 - \alpha$ credible sets, i.e. sets of $(1-\alpha)$ \emph{posterior} probability, coincide asymptotically with the sets of $(1-\alpha)$ probability under the distribution $\mathcal{N} \left(\theta^\ast_N\,, \frac{1}{N}\l(I(\theta_0)\r)^{-1}\right)$, which is well-known to be an asymptotically valid $(1-\alpha)$ ``frequentist'' confidence interval for $\theta_0$, again under minimal regularity assumptions\footnote{For instance, these are rigorously defined in the book by \citet{van2000asymptotic}.}. It is well known however that the standard posterior distribution is, in general, not robust: if the sample contains even one corrupted observation, referred to as an ``outlier'' in what follows, the posterior distribution can concentrate arbitrarily far from the true parameter $\theta_0$ that defines the data-generating distribution. Concrete scenarios showcasing this fact are given in \cite{baraud2020robust} and \cite{owhadi2015brittleness}; another illustration is presented below in example \ref{ex:2}. 
The approach proposed below addresses this drawback: the resulting posterior distribution (a) admits natural MCMC-type sampling algorithms and (b) satisfies quantifiable robustness guarantees as well as a version of the Bernstein-von Mises theorem that is similar to its classical counterpart in the outlier-free setting, where the credible sets associated with the posterior distribution are asymptotically valid confidence intervals (in the well-specified model scenario). 

Many existing works are devoted to robustness of Bayesian methods, and we attempt to give a (necessarily limited) overview of the state of the art. The papers by \citet{doksum1990consistent} and \citet{hoff2007extending} investigated approaches based on ``conditioning on partial information,'' while a more recent work by \citet{miller2015robust} introduced the notion of the ``coarsened'' posterior; however, non-asymptotic behavior (such as contraction rates) of these methods in the presence of outliers has not been explicitly addressed. Another line of work on Bayesian robustness models contamination by either imposing heavy-tailed likelihoods, like the Student's t-distribution, on the outliers \citep{Svensen2005Robust-Bayesian00}, or by attempting to identify and remove them, as was done by \citet{Bayarri1994Robust-Bayesian00}. 

As we mentioned above, the approach followed in this work relies on a version of the median of means (MOM) principle to construct a robust proxy for the log-likelihood of the data and, consequently, a robust version of the posterior distribution. 
The idea of replacing the empirical log-likelihood of the data by its robust proxy appeared previously the framework of general Bayesian updating described by \cite{bissiri2016general}, where, given the data and the prior, the posterior is viewed as the distribution minimizing the loss expressed as the sum of a ``loss-to-data'' term and a ``loss-to-prior'' term. In this framework, \citet{jewson2018principles} adopted different types of f-divergences (such as the one corresponding to the Hellinger distance), to the loss-to-prior term to obtain a robust analogue of the posterior; this approach has been investigated further in \cite{knoblauch2019generalized}. Asymptotic behavior of related types of posteriors was studied by \cite{miller2021asymptotic}, though the framework in this paper is not limited to the parametric models; when applied in the classical parametric setting, the imposed regularity conditions turn out to be more restrictive than the ones required in the present work. Various extensions for this class of algorithms were suggested, among others, by \citet[][based on so-called ``robust disparities'']{hooker2014bayesian}, \citet[][]{ghosh2016robust}, 
\cite{nakagawa2020robust}, \citet[][]{bhattacharya2019bayesian}, all of which relied on $f$-divergences, \citet[][who used kernel Stein discrepancies in place of the log-likelihood]{matsubara2021robust}, and \citet[][based on the general ``scoring rules'']{pacchiardi2021generalized}. Yet another interesting idea for replacing the log-likelihood by its robust alternative, yielding the so-called ``$\rho$-Bayes'' posterior, was proposed and rigorously investigated by \cite{baraud2020robust}. However, sampling from the $\rho$-posterior appears to be computationally difficult.

In frequentist statistics, robustness is often understood as continuity of statistical functionals with respect to the total variation distance \cite{huber2004robust}. The adversarial contamination framework described in the introduction is a natural example of the perturbation in total variation norm. At the same time, model misspecification in Bayesian statistics is often treated with respect to the Kullback-Leibler divergence \citep{kleijn2012bernstein}, which is a much stronger notion of closeness between the probability measures; for example, this issue has been discussed by \citet{owhadi2015brittleness}. In the last decade, a number of papers made successful attempts at relaxing the assumptions for misspecification using the notions of $f$-divergences, most notably, the Helliger distance \cite[][section 4.1]{jewson2018principles}, \cite{hooker2014bayesian}, \cite[][which explicitly addresses the rates of contraction]{baraud2020robust}. Since the Hellinger distance can be bounded from above and below by the total variation distance \cite{tsyb3}, one can deduce meaningful results for the latter. The coarsened posterior framework of \citet{miller2015robust} can be defined with respect to the total variation distance, however, a practical version that has been developed still relies on the relative entropy. Our proposed approach, described in detail in the following section, allows one to control the contraction rates explicitly in terms of the number of corrupted observations, albeit, only in the parametric setting; see Theorem \ref{th:1} below for the formal statement. 
\subsection{Proposed approach.}

Let $\theta'\in \Theta$ be an arbitrary fixed point in the relative interior of $\Theta$. Observe that the density of the posterior distribution $\frac{d\Pi_N(\theta\mid\m X_N)}{d\theta}$ is proportional to
$\pi(\theta)e^{-N\frac{\sum_{j=1}^N \l(\ell(\theta, X_j) -\ell(\theta',X_j)\r)}{N}}$; indeed, this is evident from equation \ref{eq:posterior} once the numerator and the denominator are divided by $\prod_{j=1}^n p_{\theta'}(X_j)$.  
The key idea is to replace the average $N^{-1}\sum_{j=1}^N \l(\ell(\theta, X_j) - \ell(\theta',X_j)\r)$
by its robust proxy denoted $\wh L(\theta)$ \footnote{Since $\theta'$ is fixed, we will suppress the dependence on $\theta'$ in the notation for brevity.} and defined formally in display \eqref{eq:lhat} below, which gives rise to the robust posterior distribution
\begin{equation}
\label{eq:rob-posterior}
\wh \Pi_N(B\lvert\m X_N) = \frac{\int_B \exp{-N\wh L(\theta)} \pi(\theta)d\theta}{\int_\Theta  \exp{-N\wh L(\theta)} \pi(\theta)d\theta}
\end{equation}
defined for all measurable sets $B\subseteq \Theta$. 
\begin{remark}
While it is possible to work with the log-likelihood $\ell(\theta,X)$ directly, it is often easier and more natural to deal with the increments $\ell(\theta, X) -\ell(\theta',X)$. For instance, in the Gaussian regression model with $X=(Y,Z)\in \mb R\times \mb R^d$, $Y=\theta^T Z+\eps$ with likelihood $p_{\theta}(y,z)\propto \exp{-\frac{(y-\theta^T z)^2}{\sigma^2}}\exp{-\frac{z^T \Sigma z}{2}}$ and $\theta'=0$, $\ell(\theta,(Y,Z)) - \ell(\theta',(Y,Z)) = \l( \theta^T Z \r)^2 - 2Y\cdot \theta^T Z$ which is more manageable than $\ell(\theta,(Y,Z))$ itself: in particular, the increments do not include the terms involving $Y^2$.
\end{remark}
Note that the density of $\wh \Pi_N(B \lvert \m X_N)$ is maximized for $\wh\theta_N = \argmin_{\theta \in \Theta} \wh L(\theta) - \frac{1}{N}\log \pi(\theta)$. For instance, if the prior $\Pi$ is the uniform distribution over $\Theta$, then $\wh \theta_N = \argmin_{\theta \in \Theta} \wh L(\theta)$ corresponds exactly to the robust risk minimization problem which, as we've previously mentioned, is hard due to non-convexity of the function $\wh L(\theta)$. At the same time, sampling from $\wh \Pi_N(B\lvert\m X_N)$ is possible, making the ``maximum a posteriori'' (MAP) estimator $\wh\theta$ as well as the credible sets associated with $\wh \Pi_N(B\lvert\m X_N)$ accessible. 
The robust risk estimator $\wh L(\theta)$ employed in this work is based on the ideas related to the \textit{median of means} principle. The original MOM estimator was proposed by \cite{nemirovsky1983problem} and later by \cite{jerrum1986random,alon1999space}. Its versions and extensions were studied more recently by many researchers including \cite{audibert2011robust,lerasle2011robust, brownlees2015empirical, lugosi2016risk, lecue2020robust2,minsker2020asymptotic}. 
We refer the reader to the survey \citep{lugosi2019mean} for a more detailed literature overview. 

Let $k \leq N / 2$ be a positive integer and $\l\{G_1, G_2, \ldots, G_k\r\}$ be $k$ disjoint subsets (``blocks'') of $\{1,2,\ldots,N\}$ of equal cardinality $\lvert G_j\rvert = n \geq N/k$, $j \in \{1,2,\ldots,k\}$. For every $\theta \in \Theta$, define the block average
\[
\Bar{L}_j(\theta) = \frac{1}{n}\sum_{i \in G_j}( \ell(\theta, X_i) - \ell(\theta',X_i)),
\]
which is the (increment of) empirical log-likelihood corresponding to the subsample indexed by $G_j$. Next, let $\rho : \mb{R} \mapsto \mb{R}^+$ be a convex, even, strictly increasing smooth function with bounded first derivative; for instance, a smooth (e.g. convolved with an infinitely differentiable kernel) version of the Huber's loss $H(x) = \min\l(x^2/2, \vert x \vert-1/2\r)$ is an example of such function. Furthermore, let $\{\Delta_n\}_{n\geq 1}$ be a non-decreasing sequence such that $\Delta_n\to\infty$ and $\Delta_n = o(\sqrt{n})$. \footnote{As a concrete example, it suffices to choose $\Delta_n\asymp \log(\log(n))$ for the results of the paper to hold. We discuss practical aspects of setting $\Delta_n$ in section \ref{section:numerical}.} Finally, define
\begin{equation}
\label{eq:lhat}
\wh L(\theta) := \argmin_{z \in \mb{R}} \sum_{j=1}^k\rho\l(\sqrt{n}\frac{\Bar{L}_j(\theta) - z}{\Delta_n}\r),
\end{equation}
which is clearly a solution to the convex optimization problem. 
Robustness and non-asymptotic performance of $\wh L(\theta)$ can be quantified as follows. Let $\sigma^2(\theta) = \mathrm{var}\l(\ell(\theta,X)-\ell(\theta',X)\r)$, and $\wt \Delta_n = \max(\sigma(\theta), \Delta_n)$; then for all $s$, and number of outliers $\m O$ such that $\max(s,\m O)\leq ck$ for some absolute constant $c>0$, 
\begin{align}
\label{eq:deviations1}
\l\lvert \wh L(\theta) - L(\theta) \r\rvert \leq \frac{\wt \Delta_n}{\Delta_n} \sigma(\theta)\sqrt{\frac{s}{N}} 
+ \wt \Delta_n\l( \frac{s+\m O}{k \sqrt n} +  \sqrt{\frac{k}{N}} o\l(1\r)\r)
\end{align}
with probability at least $1-2e^{-s}$, where $o(1)\to 0$ as $\max(\Delta_n,n)\to\infty$. Put simply, under very mild assumptions on $\ell(\theta,X)-\ell(\theta',X)$, $\wh L(\theta)$ admits sub-Gaussian deviations around $L(\theta)$, moreover, it can absorb the number of outliers that is of order $k$. We refer the reader to Theorem 3.1 in \citep{minsker2018uniform} for a uniform over $\theta$ version of this bound as well as more details. 

Let us mention that the posterior distribution $\wh{\Pi}_N$ is a valid probability measure, meaning that
$\wh \Pi_N(\Theta\lvert\m X_N) =1 $. By the definition at display \eqref{eq:rob-posterior}, it suffices to show the denominator, $\int e^{-N \wh{L}(\theta)}\pi(\theta) d\theta$, is finite.  Indeed, note that $\wh{L}(\theta) > \wh{L}(\wt{\theta}_{N})$ for all $\theta$ where $\wt{\theta}_{N} = \argmin_{\theta\in \Theta} \wh L(\theta)$, hence
\begin{align*}
    \int e^{-N \wh{L}(\theta)}\pi(\theta) d\theta \leq e^{-N \wh{L}(\wt{\theta}_{N})}\int_\Theta \pi(\theta) d\theta.
\end{align*}
Therefore, a sufficient condition for $\int_\Theta e^{-N \wh{L}(\theta)} \pi(\theta)d\theta$ being finite is 
$\wh{L}(\wt{\theta}_{N}) > -\infty$ a.s. This is guaranteed by the fact that under mild regularity assumptions, for any $\theta, \theta' \in \Theta$, $\ell(\theta, x) - \ell(\theta', x)  > -\infty$, $P_{\theta_0}$ - almost surely.  
We end this section with a remark related to the smoothness assumptions required for the loss function $\rho$.
\begin{remark}
The classical MOM estimator corresponds to the choice $\rho(x)=\vert x \vert$ which is not smooth but is scale-invariant, in a sense that the resulting estimator does not depend on the choice of $\Delta_n$. 
While the latter property is often desirable, we conjecture that the posterior distribution based on such ``classical'' MOM estimator does not satisfy the Bernstein-von Mises theorem, and that smoothness of $\rho$ is important beyond being just a technical assumption. This somewhat surprising conjecture is so far only supported by our simulation results explained in Example \ref{ex:bvm_fails} below.
\end{remark}
\begin{example}
\label{ex:bvm_fails}
Let $\mathcal{X}_N = (X_1,\ldots, X_N)$ be i.i.d. with normal distribution $\mathcal{N}(\theta, 1)$, $\theta_0 = -30$ and the prior distribution for $\theta$ is $\mathcal{N}(-29.50, 1)$. Furthermore, let $\rho(x) = \vert x \vert$. 
We sample from the robust posterior distribution for the values of $k = 20, 40,60,80 $ and $n = \lfloor 1000/k \rfloor$. The resulting plots are presented in Figure \ref{fig:bvm_fails}. The key observation is that the posterior distributions are often multimodal and skewed, unlike the expected ``bell shape.'' This behavior persists even when the sample size is increased. 
\begin{figure}[ht]
\begin{center}
    \includegraphics[width = 0.8\textwidth, height = 7cm]{bvm_fails_ICML_new.pdf}
    \caption{Posterior distribution $\hat{\Pi}_N$ for Example \ref{ex:bvm_fails}. The dark blue curve is the density function and the light blue area represents the $95\%$ credible interval.}
    \label{fig:bvm_fails}
\end{center}
\end{figure}
\end{example}

\section{Main results.}
\label{sec:main}

We are ready to state the main results related to the properties of the robust posterior distribution $\wh \Pi_N(\cdot \lvert \m X_N)$. 
First, we will state them in a way that avoids technical assumptions which can be found in the latter part of the section. We specify the exact requirements for each of our theorems in remark \ref{remark:assumptions}. Recall that $L(\theta) = \mb E\ell(\theta, X)$ where $\ell(\theta, x) = - \log p_\theta (x)$, and let $\sigma(\Theta):=\sup_{\theta\in \Theta}\mathrm{var}\l(\ell(\theta,X)\r)$ and $\wt \Delta = \max\l(\Delta_n, \sigma(\Theta)\r)$. 
Everywhere below, $\|\cdot\|$ stands for the Euclidean norm of a vector. 

The following theorem characterizes robustness properties of the posterior distribution $\wh \Pi_N(\cdot\lvert\m X_N)$ in terms the contraction rates towards the parameter $\theta_0$, as well as performance of the mode of $\wh \Pi_N(\cdot\lvert\m X_N)$ defined as 
\begin{equation}
\label{eq:mode}
\wh \theta_N=\argmin_{\theta \in \Theta} \wh L(\theta) - \frac{1}{N}\log \pi(\theta).
\end{equation}
\begin{theorem}
\label{th:1}
Under the appropriate regularity conditions on the function $\rho$, prior $\Pi$ and the family $\{p_\theta, \ \theta\in \Theta\}$, the following holds:
\begin{enumerate}
    \item 
    With probability at least $99\%$,
\[
\l\| \wh \theta_N - \theta_0 \r\|^2 \leq C\l( \wt \Delta \l( \frac{\m O+1}{k\sqrt n} + \l(\frac{k}{N}\r)^{\frac{1+\tau}{2}} \r) \r)+ O\l( \frac{1}{\sqrt N}\r)  
\]
as long as $\m O\leq ck$ for some absolute constants $c,C>0$. Here, $\tau\in (0,1]$ is a constant defined in assumption \ref{ass:3B} below.
    \item
    Let $\{\delta_N\}_{N\geq 1}$ be sequence such that $\delta_N \rightarrow 0$ and $N\delta_N^2 - \Delta_n\m O \sqrt{n}\to \infty$
    as $N \rightarrow \infty$. Moreover, assume $\mathcal{O}$ Then
    \begin{align*}
        \wh \Pi_N \l(\l\| \theta - \theta_0 \r\| \geq \delta_N \mid \mathcal{X}_N \r) \rightarrow 0
    \end{align*}
    in probability.
\end{enumerate}
\end{theorem}
The first part of the theorem implies that as long as the number of blocks containing outliers is not too large, the effect of these outliers on the estimation error is limited, regardless of their nature and magnitude; of course, the bound can be shown to hold with confidence arbitrary close to 1, instead of $99\%$, at the cost of increasing the constants. 
While the fact that the mode of $\wh \Pi_N$ is a robust estimator of $\theta_0$ is encouraging, one has to address the ability of the method to quantify uncertainty: the second part of the theorem is a result in this direction and states that $\wh \Pi_N$ ``contracts'' towards $\theta_0$, in other words, it assigns most of its mass to a small neighborhood of $\theta_0$. One the one hand, it implies that the credible sets associated with $\wh \Pi_N$ are robust. On the other hand, in the outlier-free framework ($\m O=0$), it implies that the credible sets have width of order $O(N^{-1/2})$, similar to the usual posterior distribution (indeed, in this case $\delta_N$ is required to satisfy the relations $\delta_N\to 0$ and $N\delta^2_N\to\infty$, so one can take $\delta_N = \beta_N N^{-1/2}$ where $\{\beta_N\}_{N\geq 1}$ is any increasing unbounded sequence).

Is it possible to give a better description of the asymptotic behavior of the posterior distribution, beyond specifying the rate of contraction towards $\theta_0$?  
We give a positive answer to this question in the contamination-free setting, and show that the robust posterior can ``adapt'' to such a favorable framework: specifically, it satisfies a version of the Bernstein-von Mises theorem. Let 
\begin{equation}
\label{eq:theta-tilde}
    \wt{\theta}_{N} = \argmin_{\theta\in \Theta} \wh L(\theta).
\end{equation}
It can be viewed as a robust proxy of the MLE or the mode of the posterior distribution corresponding to the uniform prior on $\Theta$.
\begin{theorem}
\label{th:2}
Assume the outlier-free framework. Under appropriate regularity conditions on the prior $\Pi$ and the family 
$\{p_\theta, \ \theta\in \Theta\}$, 
\begin{align*}
\left\|\wh{\Pi}_N(\cdot\lvert\m X_N) - \mathcal{N}\left(\wt{\theta}_{N}, \frac{1}{N}\l(\partial_\theta^2 L(\theta_0)\r)^{-1}\right)\right\|_{\mathrm{TV}} \overset{P}{\longrightarrow} 0.
\end{align*}
Moreover, $\sqrt{N}\l( \wt{\theta}_{N} - \theta_0 \r) \overset{d}{\longrightarrow} \mathcal{N} \l(0, \l(\partial^2_\theta L(\theta_0)\r)^{-1}I(\theta_0)\l(\partial^2_\theta L(\theta_0)\r)^{-1}\r)$.
\end{theorem}
The message of this result is that in the ideal, outlier-free scenario, $\wh\Pi_N$ inherits the properties of the standard posterior distribution $\Pi_N$. In particular, in the ``well-specified'' case $p = p_{\theta_0}$ for some $\theta_0\in \Theta$, 
$\partial_\theta^2 L(\theta_0) = I(\theta_0)$ under mild regularity assumptions, implying that the credible sets associated with the posterior distribution are asymptotically valid confidence intervals.
In the misspecified model scenario, that is, when $p \notin \{p_\theta : \theta \in \Theta\}$), the asymptotic covariance of $\tilde \theta_N$ is the so-called ``sandwich covariance.'' In general, it is different from the asymptotic covariance of the posterior distribution, implying that $\wh\Pi_N$ can either underestimate or overestimate the uncertainly. This behavior is consistent with performance of the standard posterior $\Pi_N$ that has been investigated and discussed in detail by \citet{kleijn2012bernstein,kleijn2006misspecification,panov2015finite}. 

The main novelty in Theorem \ref{th:2} is the first, ``BvM part,'' while the asymptotic normality of $\wt \theta_N$ has been previously established in \citep{minsker2020asymptotic}.

We finish this section by listing and discussing the complete list of regularity conditions that are required for the stated results to hold. Recall that $\|\cdot\|$ stands for the standard Euclidean norm. 
\begin{assumption}
\label{ass:1}
The function $\rho: \mb R\mapsto \mb R_+$ is convex, even, and such that
\begin{enumerate}
\item[(i)] $\rho'(z)=z$ for $\vert z \vert \leq 1$ and $\rho'(z)=\mathrm{const}$ for $\vert z \vert \geq 2$.
\item[(ii)] $z - \rho'(z)$ is nondecreasing on $R_+$; 
\item[(iii)] $\rho^{(5)}$ is bounded and Lipschitz continuous.
\end{enumerate}
\end{assumption}
One example of such $\rho$ is the smoothed Huber's loss: let
\begin{align*}
    H(z) = \frac{z^2}{2}\mb{I}\l\{\vert z \vert\leq 3/2\r\} + \frac{3}{2}\l(\vert z \vert - \frac{3}{4}\r)\mb{I} \l\{\vert z \vert > 3/2\r\}\,.
\end{align*}
Moreover, set  $\psi(z) = Ce^{-\frac{4}{1 - 4z^2}}\mb{I}\l\{\vert z \vert \leq \frac{1}{2}\r\}$. 
Then $\rho(z) = (H \star \psi)(z)$, where $\star$ denotes the convolution, satisfies assumption \ref{ass:1}. 
Condition (iii) on the higher-order derivatives is technical in nature and can likely be avoided at least in some examples; in numerical simulations, we did not notice the difference between results based on the usual Huber's loss and its smooth version. Next assumption is a standard requirement related to the local convexity of the loss function $L(\theta)$ at its global minimum $\theta_0$. 
\begin{assumption}
\label{ass:2}
The Hessian $\pd^2_\theta L(\theta_0)$ exists and is strictly positive definite. 
\end{assumption}
In particular, this assumption ensures that in a sufficiently small neighborhood of $\theta_0$, $c(\theta_0)\|\theta - \theta_0\|^2 \leq L(\theta) - L(\theta_0) \leq C(\theta_0)\|\theta - \theta_0\|^2$ for some constants $0 < c(\theta_0) \leq C(\theta_0) < \infty$. 
The following two conditions allow one to control the ``complexity'' of the class $\{\ell(\theta,\cdot), \ \theta\in\Theta\}$. 
\begin{assumption}
\label{ass:3}
For every $\theta\in \Theta$, the map $\theta'\mapsto\ell(\theta',x)$ is differentiable at $\theta$ for $P$-almost all $x$ (where the exceptional set of measure $0$ can depend on $\theta$), with derivative $\pd_\theta \ell(\theta,x)$. Moreover, $\forall \theta\in \Theta$, the envelope function
$\m V(x{;}\delta):=\sup_{\|\theta' - \theta\|\leq\delta} \l\| \pd_\theta \ell(\theta',x)\r\|$ of the class $\l\{ \pd_\theta \ell(\theta',\cdot): \|\theta'-\theta\|\leq \delta\r\}$ satisfies $\mb E \m V^{2+\tau}(X;\delta)<\infty$ for some $\tau\in(0,1]$ and a sufficiently small $\delta=\delta(\theta)$. 
\end{assumption}
An immediate implication of this assumption is the fact that the function $\theta\mapsto \ell(\theta,x)$ is locally Lipschitz. It other words, for any $\theta\in \Theta$, there exists a ball $B(\theta,r(\theta))$ of radius $r(\theta)$ such that for all $\theta', \, \theta''\in B(\theta,r(\theta))$ 
\[
\lvert \ell(\theta',x) - \ell(\theta'',x)\rvert \leq \m V(x{;}\delta)\|\theta' - \theta''\|.
\]
In particular, this condition suffices to prove consistency of the estimator $\wt\theta_N$. 

The following condition is related to the modulus of continuity of the empirical process indexed by the gradients $\partial_\theta \ell(\theta,x)$. It is similar to the typical assumptions required for the asymptotic normality of the MLE, such as Theorem 5.23 in the book by \citet{van2000asymptotic}. Define
\begin{align*}
    \omega_N(\delta) = \mb E \sup_{\|\theta - \theta_0\| \leq \delta } \l\| \sqrt{N}  \l(P_N - P\r)(\partial_\theta \ell(\theta, \cdot) - \partial_\theta \ell(\theta_0, \cdot))\r\|\,,
\end{align*}
where $P_N$ is the empirical distribution by $\tilde{\m X}_N$.
\begin{assumption}
\label{ass:3B}
The following relation holds: 
\[
\lim\limits_{\delta\to 0}\limsup_{N\to\infty} \omega_N(\delta) = 0.
\]
Moreover, the number of blocks $k$ satisfies $k=o(n^\tau)$ as $k,n\to\infty$, where $\tau\in (0,1]$ is the same constant as in assumption \ref{ass:3}.
\end{assumption}
Limitation on the growth of $k$ is needed to ensure that the bias of the estimator $\wh L(\theta)$ is of order $o(N^{-1/2})$, a fact that we rely on in the proof of Theorem \ref{th:2} as well second part of Theorem \ref{th:1}. In many scenarios, $\tau$ can be chosen to be $1$, whence the requirement $k=o(n)$ is equivalent to $k=o\l(\sqrt{N}\r)$. 
Finally, we state a mild requirement imposed on the prior distribution; it is only slightly more restrictive than its counterpart in the classical BvM theorem (for example, Theorem 10.1 in the book by \citet{van2000asymptotic}).
\begin{assumption}
\label{ass:5}
The density $\pi$ of the prior distribution $\Pi$ is positive and bounded on $\Theta$, and is continuous on the set $\{\theta: \|\theta - \theta_0\| < c_{\pi}\}$ for some positive constant $c_{\pi}$.
\end{assumption}
\begin{remark}
\label{remark:assumptions}
Part (1) of Theorem \ref{th:1} relies on all assumption besides assumption \ref{ass:3B}, while part (2) of Theorem \ref{th:1} and Theorem \ref{th:2} require all assumptions to hold. 
\end{remark}
\begin{remark}
Most commonly used families of distributions satisfy assumptions \ref{ass:2}-\ref{ass:3B} with $\tau=1$. For example, this is easy to check for the normal, Laplace or  Poisson families in the location model where $p_\theta(x) = f(x-\theta), \ \theta \in \Theta$. Other examples covered by our assumptions include linear regression with Gaussian or Laplace-distributed noise. The main work is usually required to verify assumption \ref{ass:3B}; it relies on the standard tools for the bounds on empirical processes for classes that are Lipschitz in parameter (this case is covered by Lemma \ref{lemma:sup-power} in the appendix) or have finite Vapnik-Chervonenkis dimension. Examples of the latter can be found in the books by \citet{van2000asymptotic} and \citet{van1996weak}.
\end{remark}

\section{Numerical examples and applications.}
\label{section:numerical}

We will illustrate our theoretical findings by presenting numerical examples below. The loss function that we use is Huber's loss defined before. While, strictly speaking, it does not satisfy the smoothness requirements, we found that it did not make a difference in our simulations. Algorithm for sampling from the posterior distributions was based on the ``No-U-Turn sampler'' variant of Hamiltonian Monte Carlo method \citep{hoffman2014no}. Robust estimator of the log-likelihood $\wh{L}(\theta)$ are approximated via the gradient descent algorithm at every $\theta$. 

The parameter $\Delta_n$ that is required in the construction of the estimator $\Delta_n$ is a proxy for the standard deviation of $\ell(\theta,X) - \ell(\theta',X)$. The initial choice $\Delta_n:=\Delta_{n,0}$ is a fixed number (e.g., $\Delta_{n,0} = 1)$, and $\theta'$ is also chosen arbitrarily. Given these choices and an initial sample $\wt\theta_1,\ldots,\wt\theta_t$ from the posterior $\wh\Pi_N$ (for instance, $t=50$), we set $\theta'$ to be the geometric median of this sample, that is, 
$\theta' = \argmin_{\theta\in \Theta} \sum_{i=1}^t \|\theta-\wt\theta_i\|$.  
Then we update $\Delta_n$ via computing a robust estimator, namely, the median absolute deviation, of the standard deviation of $\ell(\theta,X)- \ell(\theta',X)$. To this end, set $\widehat M(\theta):=\mathrm{median}\l(\bar L_1(\theta),\ldots,\bar L_k(\theta) \r)$, and
\[
\mathrm{MAD}(\theta) = \mathrm{median}\l( \lvert \bar L_1(\theta) -\wh M(\theta)\rvert, \ldots,\lvert \bar L_k(\theta) - \widehat M(\theta)\rvert \r).
\]
Finally, define $\wh \Delta_{n} := \frac{\mathrm{MAD}(\theta}{\Phi^{-1}(3/4)}$, where $\Phi$ is the distribution function of the standard normal law and the normalizing factor comes from the fact that for a sample from the normal distribution $N(\mu,\sigma^2)$, the expected value of $\mathrm{MAD}$ equals $\Phi^{-1}(3/4)\sigma$. The data-dependent choice $\wh\Delta_n$ works well in numerical simulations.

In both examples below, we use Gaussian likelihoods that satisfy assumptions \ref{ass:2}-\ref{ass:3B} required by our theory. The following example provides empirical evidence of the fact that using Huber's loss in the framework of example \ref{ex:bvm_fails} is enough for BvM theorem to hold. 
\begin{example}
\label{ex:2}
We consider two scenarios: in the first one, the data are $N = 1000$ i.i.d. copies of $\mathcal{N}(-30, 1)$ random variables. In the second scenario, data are generated in the same way except that $40$ randomly chosen observations are replaced with $40$ i.i.d. copies of $\mathcal{N}(10^4,1)$ distributed random variables. Results are presented in figures 
\ref{fig:correct_hat_pi_no_outlier} and \ref{fig:correct_hat_pi_outlier}, where the usual posterior distribution is plotted as well for comparison purposes. 
The main takeaway from this simple example is that the proposed method behaves as expected: as long as the number of blocks $k$ is large enough, robust posterior distribution concentrates most of its ``mass'' near the true value of the unknown parameter, while the usual posterior distribution is negatively affected by the outliers. 
At the same time, in the outlier-free regime, both posteriors perform similarly. 
\begin{figure}[ht]
\begin{center}
    \includegraphics[width = 0.8\textwidth, height = 5cm]{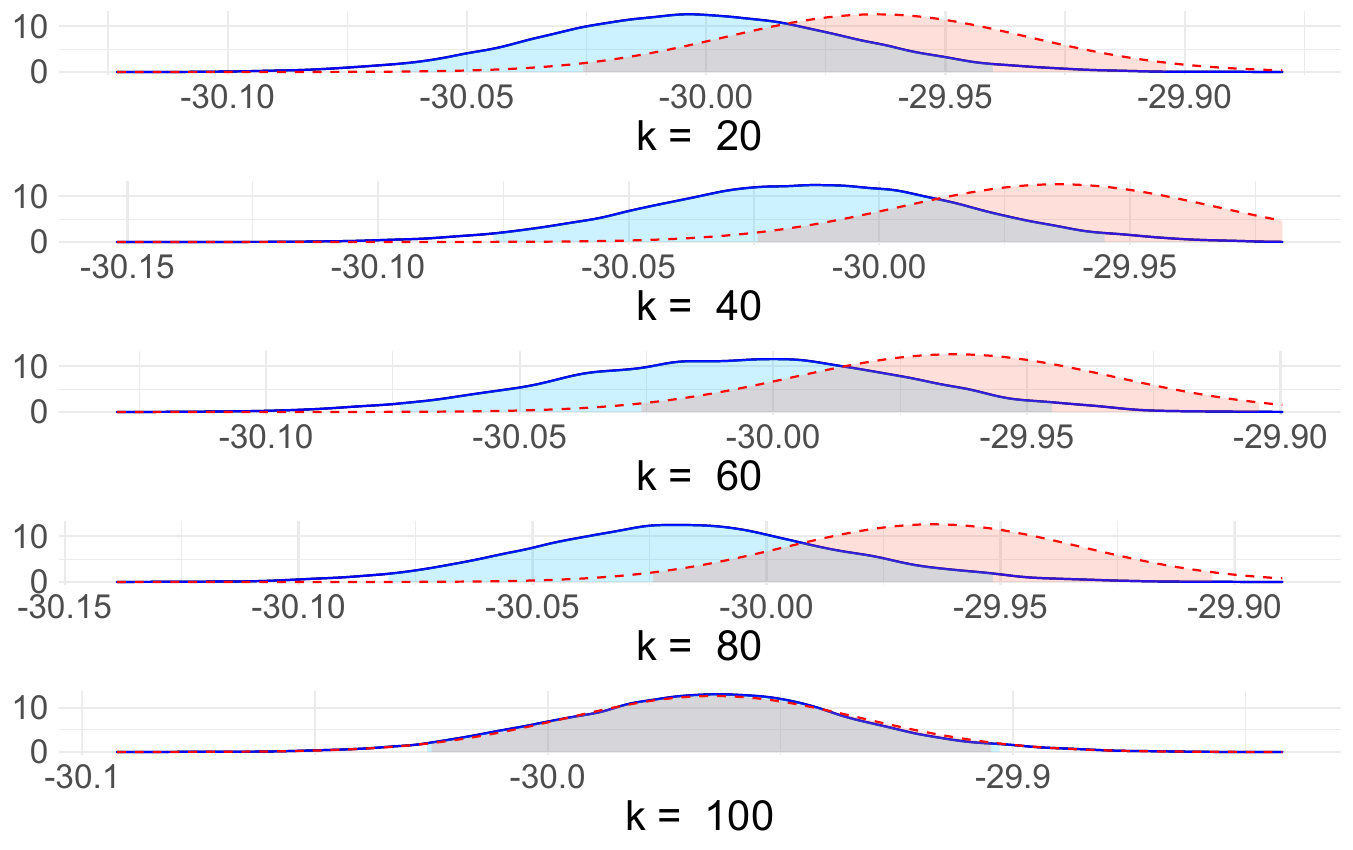}
    \caption{Posterior distribution $\hat{\Pi}_N$ for Example \ref{ex:2}, scenario 1. The blue curves and blue shaded regions correspond to the density function and $95\%$ credible sets of $\hat{\Pi}_N$ whereas dashed red curves and red shaded region are the standard posterior and its corresponding $95\%$ credible set.}
    \label{fig:correct_hat_pi_no_outlier}
    \vspace{1pc}
      \includegraphics[width = 0.8\textwidth, height = 5cm]{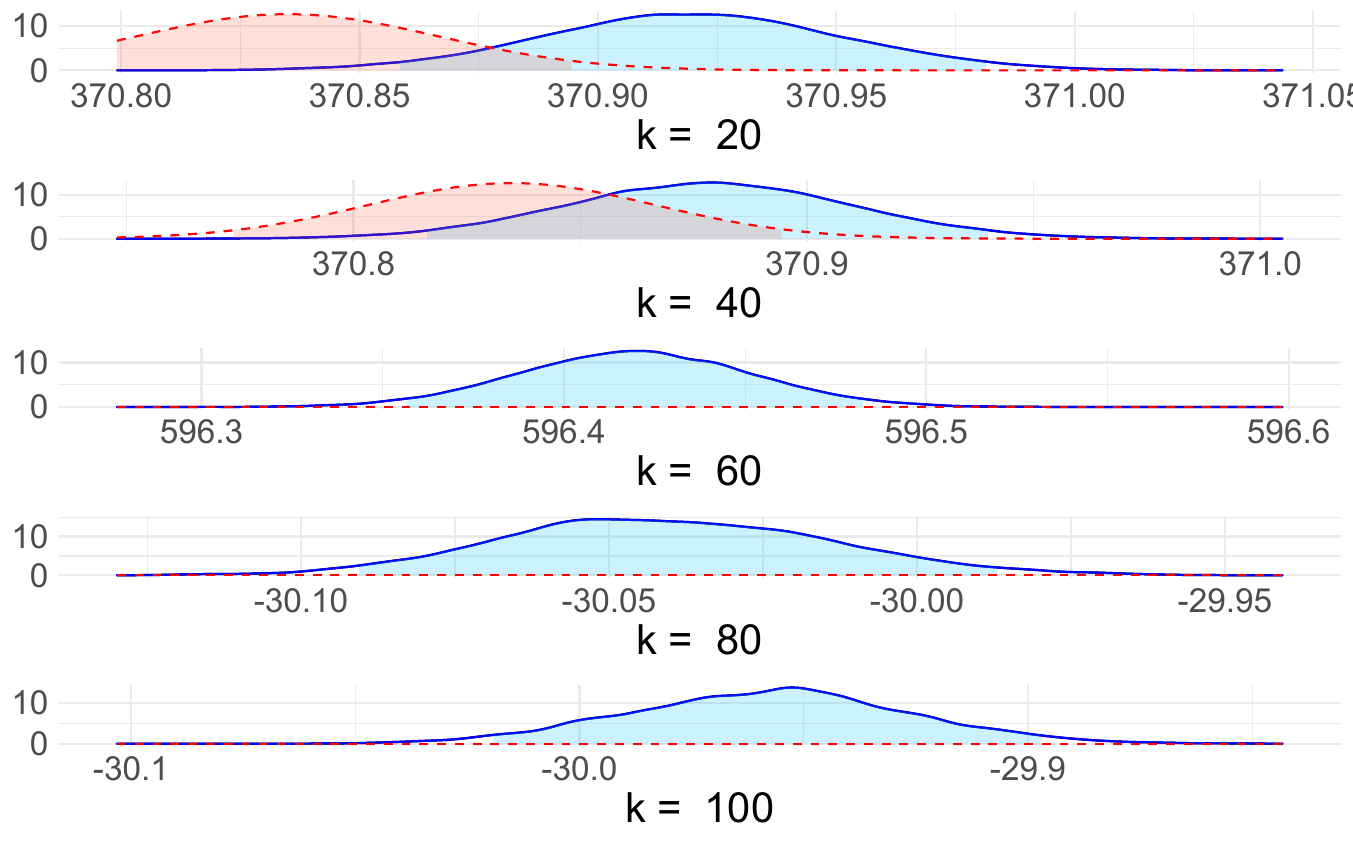}
    \caption{Posterior distribution $\hat{\Pi}_N$ for Example \ref{ex:2}, scenario 2.}
    \label{fig:correct_hat_pi_outlier}
\end{center}
\end{figure}
\end{example}
\begin{example}
\label{ex:3}
In this example, we consider a non-synthetic dataset in the linear regression framework. 
The dataset in question was provided by \cite{cortez2009modeling} and describes the qualities of different samples of red and white wine. It contains $11$ ``subjective'' features such as fixed acidity, pH, alcohol, etc., and one ``objective'' feature, the scoring of wine quality; $4898$ white wine samples are selected to perform the linear regression where the objective feature is the response and the subjective features are the regressors. It is assumed that the data is sampled from the model 
\begin{align*}
        Y = \beta_0 + \beta_1X_1 + \beta_2X_2 + \ldots \beta_{8}X_{8} + \varepsilon\,,
\end{align*}
where $\beta_0$ is the intercept, $Y$ is the response, $X_1, X_2, \ldots, X_8$ are the chosen regressors (see detailed variable names in Table \ref{tbl:real_data}) along with the corresponding coefficients $\beta_1, \beta_2, \ldots, \beta_8$, and $\varepsilon$ is the random error with $N(0, \sigma^2)$ distribution. Here we remark that, for simplicity only $8$ out of $11$ ``subjective'' features are selected such that this model agrees with the OLS linear regression model generated by best subset selection with minimization of BIC. 
In the second experiment, $10$ randomly chosen response variables are replaced with $\mathcal{N}(1000, 10)$ random outliers. 
In both cases, the priors for $\beta_j$ are set to be $N(0, 10^2)$ for every $j$, and the prior for $\sigma$ is the uniform distribution on $(0, 1]$. The block size $n$ is set to be $158$ and the number of blocks $k$ is $31$. The MAP estimates of $\beta_j$'s and $\sigma$, as well as the two end points of the $95\%$ credible intervals are reported in Table \ref{tbl:real_data}. These plots yet again demonstrate that the posterior $\wh \Pi_N$, unlike its standard version, shows stable behavior when the input data are corrupted.
\begin{table}[h]
\caption{MAP estimates of the intercept, regression coefficients and the standard deviation $\sigma$, left and right end points of $95\%$ credible intervals in parentheses.}
\label{tbl:real_data}
\centering
\begin{tabular}{|c|c|}
\hline
\multicolumn{1}{|c|}{variable name} & \multicolumn{1}{c|}{$\wh{\Pi}_N$} \\ \hline
intercept & $-0.002(-0.026, 0.022)$ \\ \hline
fixed.acidity & $0.065(0.027, 0.101)$ \\ \hline
volatile.acidity & $-0.207(-0.232, -0.183)$ \\ \hline
residual.sugar & $0.453(0.381, 0.536)$ \\ \hline
free.sulfur.dioxide & $0.077(0.051, 0.102)$ \\ \hline
density & $-0.487(-0.600, -0.372)$ \\ \hline
pH & $0.125(0.093, 0.160)$ \\ \hline
sulphates & $0.075(0.049, 0.101)$ \\ \hline
alcohol & $0.287(0.227, 0.350)$ \\ \hline
$\sigma$ & $0.852(0.835, 0.868)$ \\ \hline
\end{tabular}
\end{table}
\begin{figure}[ht]
\begin{center}
    \includegraphics[width = 0.8\textwidth, height = 7.5cm]{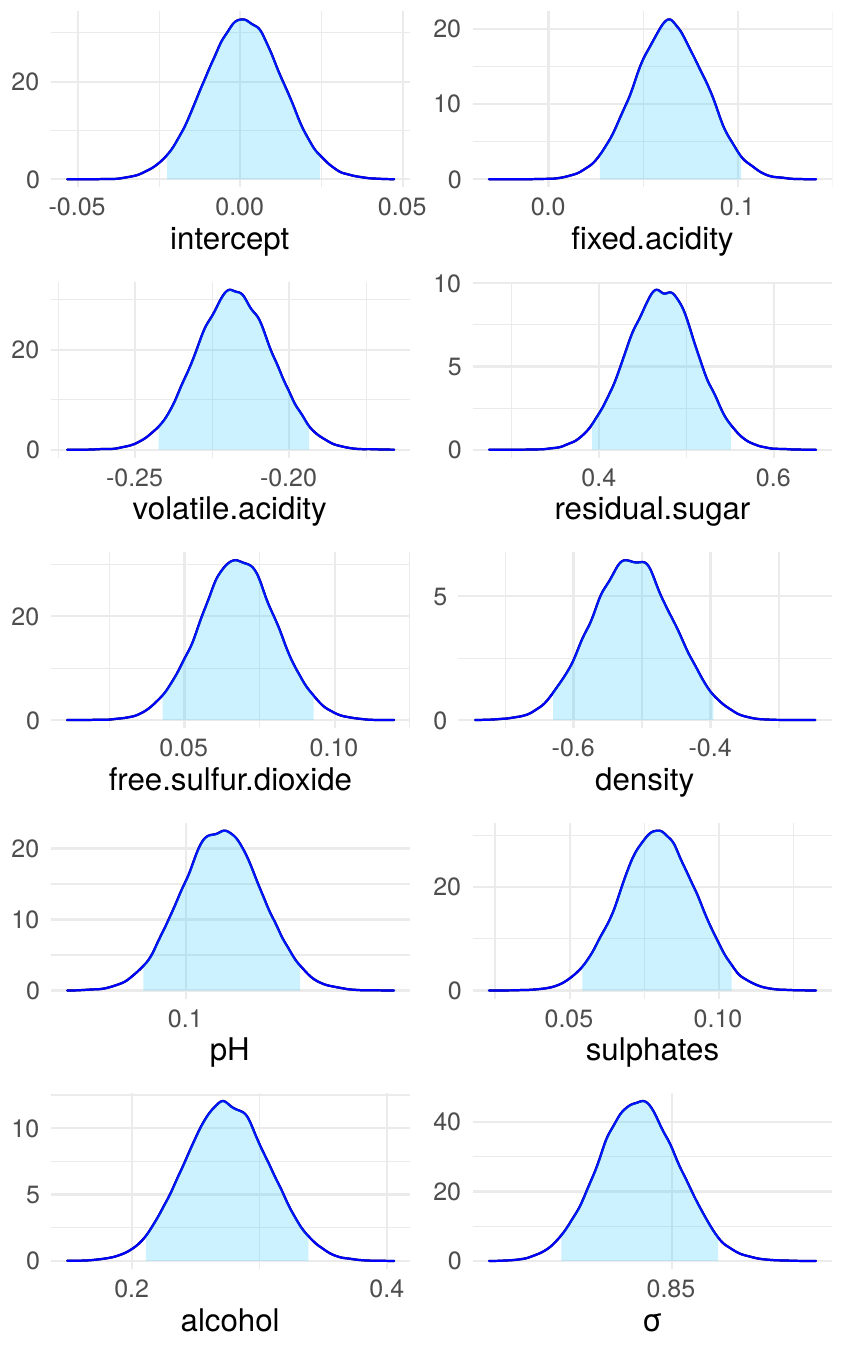}
    \caption{Posterior distribution $\wh{\Pi}_N$ for Example \ref{ex:3}, no outliers.}
    \label{fig:real_data_1}
    \includegraphics[width = 0.8\textwidth, height = 7.5cm]{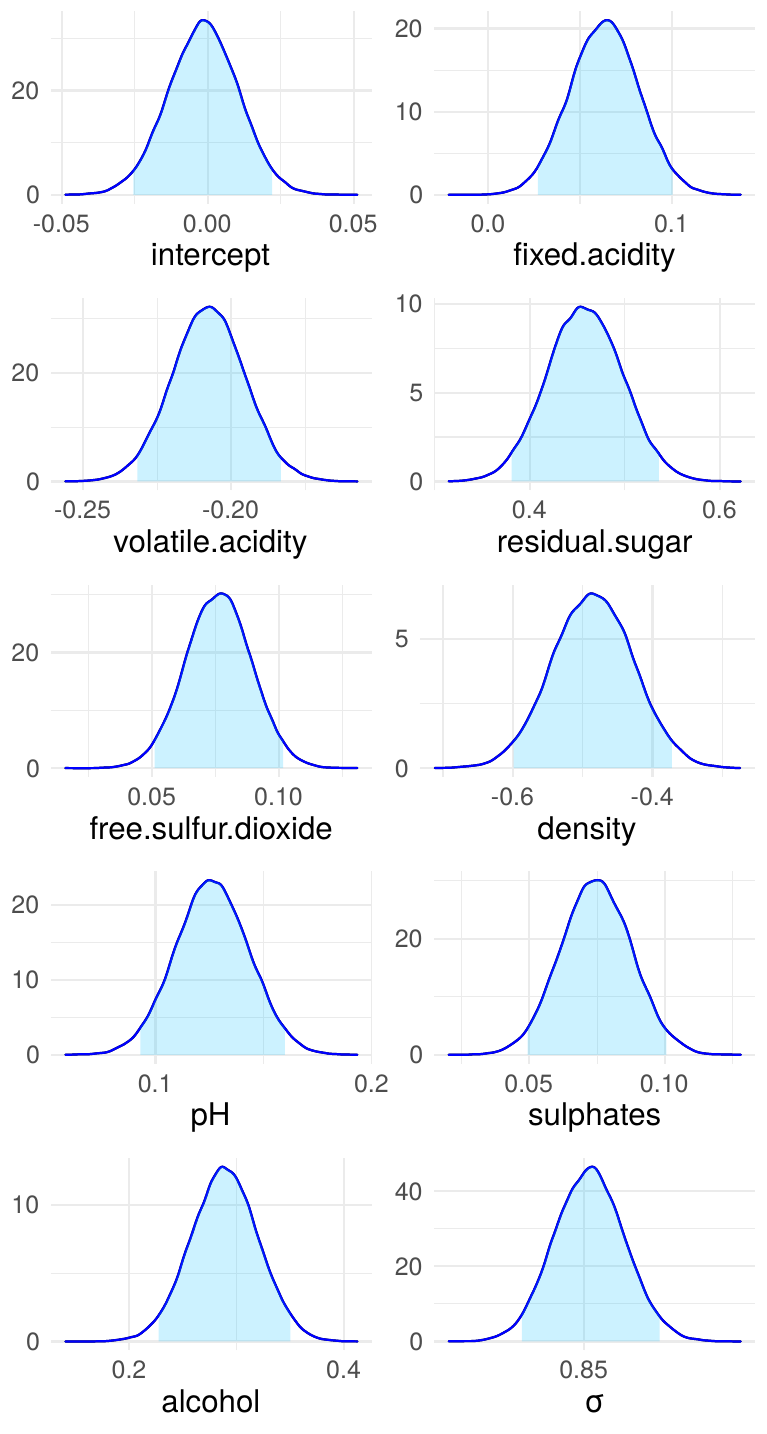}
    \caption{Posterior distribution $\wh{\Pi}_N$ for Example \ref{ex:3}, with outliers.}
    \label{fig:real_data_1_outlier}
\end{center}
\end{figure}
\begin{figure}
\begin{center}
  \includegraphics[width = 0.8\textwidth, height = 8cm]{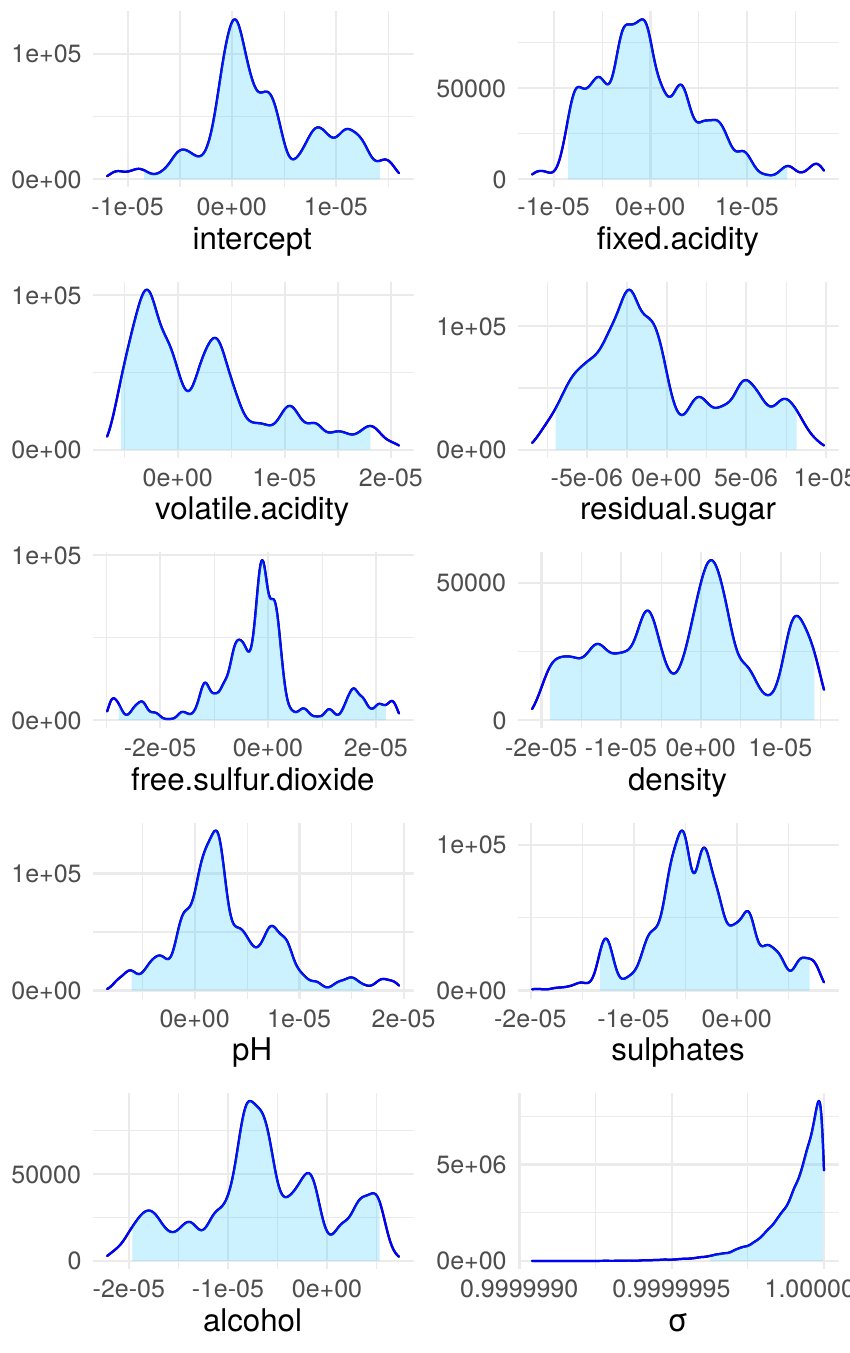}
    \caption{Standard posterior distribution for Example \ref{ex:3}, with outliers.}
    \label{fig:normal_outlier}
    \end{center}
\end{figure}
\end{example}


\section{Discussion.}
\label{sec:discussion}

The proposed extension of the median of means principle to Bayesian inference yields a version of the posterior distribution possessing several desirable characteristics, such as (a) robustness, (b) valid asymptotic coverage and (c) computational tractability. In addition, the mode of this posterior distribution serves as a robust alternative to the maximum likelihood estimator. The computational cost of our method is higher compared to the usual posterior distribution as we need to solve a one-dimensional convex optimization problem to estimate the expected log-likelihood, however, the method is still practical and, unlike many existing alternatives with similar theoretical properties, can be implemented with many off the shelf sampling packages. As with many MOM-based methods, the main ``tuning parameter'' is the number of blocks $k$: while larger $k$ increases robustness, smaller values $k$ reduce the bias in the estimation of the likelihood. In many examples however, this bias is far from the worst case scenario, and we observed that in our simulations, the method behaves well even when the size of each ``block'' is small. As a practical rule of a thumb, we recommend setting $k\asymp \sqrt{N}$ if no prior information about the number of outliers is available: this stems from the fact that in typical scenarios, this choice guarantees that the bias of the estimator $\wh L(\theta)$ does not dominate the stochastic error; see the discussion immediately following assumption \ref{ass:3B} for more details. If an upper bound $\m O$ on the number of outliers is known, then one should set $k = 2\m O$ to ensure that at least half of the blocks of the data will be contamination-free. As for the choice of the sequence $\Delta_n$, it suffices to set $\Delta_n = \log\log(n)$ for the main theoretical guarantees to hold. As for the practical choice, we provided additional guidance in section \ref{section:numerical}.

Now, let us discuss the drawbacks. First of all, the requirement for $\Theta$ to be compact is quite restrictive, and is typically necessary to ensure that the quantity $\sigma(\Theta)=\sup_{\theta\in \Theta}\mathrm{var}\l( \ell(\theta,X)\r)$ appearing in our bounds is finite. This root of this requirement is related to the fact that $\wh L(\theta)$, viewed as an estimator of the mean, is not scale-invariant. At the same time, compactness assumption is satisfied if one has access to some preliminary, ``low-resolution'' proxy $\tilde\theta$ of $\theta_0$ such that $\|\theta_0-\tilde\theta\|\leq R$ for some, possibly large, $R>0$. Second, our method is currently tailored only for the case of i.i.d. data and the parametric models, which is the most natural setup that is natural for demonstrating the ``proof of concept.'' At the same time, it would be interesting to obtain practical and theoretically sound extensions that are applicable in more challenging frameworks.

\section{Proofs.}

This section explains the key steps behind the proofs of our mains results. The complete argument leading to Theorem \ref{th:2} is rather long and technical. Here, we will outline the main ideas of the proof and the reduction steps that are needed to transform the problem into an easier one, while the missing details are included in the appendix.

\section{Proof of Theorem \ref{th:1}, Part (1).}
We will show part 1 first. In view of assumption \ref{ass:2}, $\|\theta-\theta_0\|^2 \leq c'(\theta_0) \l( L(\theta) - L(\theta_0)\r)$ whenever $\|\theta - \theta_0\|$ is sufficiently small. Hence, if we show that $\wh\theta_N$ satisfies this requirement, we would only need to estimate $L(\wh\theta_N) - L(\theta_0)$. To this end, denote $L(\theta,\theta') = L(\theta) - L(\theta')$, and observe that 
\begin{multline*}
L(\wh\theta_{N},\theta') = L(\wh\theta_{N},\theta') - \hL(\wh\theta_{N}) + \hL(\wh\theta_{N}) + \frac{1}{N}\log(1/\pi(\wh\theta_N)) -\frac{1}{N}\log(1/\pi(\wh\theta_N))
\\
\leq L(\wh\theta_{N},\theta') - \hL(\wh\theta_{N}) + \hL(\theta_0) + \frac{1}{N}\log\l(\frac{\pi(\wh\theta_N)}{\pi(\theta_0)}\r)
\\
\leq L(\theta_0,\theta') + 2\sup_{\theta\in \Theta}\l\lvert  L(\theta,\theta') - \hL(\theta)\r\rvert  + \frac{1}{N}\log\l(\frac{\pi(\wh\theta_N)}{\pi(\theta_0)}\r).
\end{multline*}
If $\pi(\wh\theta_N) \leq \pi(\theta_0)$, the last term above can be dropped without changing the inequality. On  the other hand, $\pi(\theta)$ is bounded and $\pi(\theta_0)>0$, $\frac{\pi(\wh\theta_N)}{\pi(\theta_0)}\leq \frac{\|\pi\|_\infty}{\pi(\theta_0)}$, whence the last term is at most $\frac{C(\pi,\theta_0)}{N}$. 
Given $\eps>0$, assumption \ref{ass:2} implies that there exists $\delta>0$ such that $\inf_{\|\theta-\theta_0\|\geq \eps} L(\theta) >L(\theta_0)+\delta$. Let $N$ be large enough so that $\frac{C(\pi,\theta_0)}{N}\leq \delta/2$, whence 
$\pr{\|\wh\theta_{N} - \theta_0\|\geq \eps} \leq \pr{\sup_{\theta\in \Theta}\lvert \hL(\theta) - L(\theta,\theta')\rvert >\delta/2 }$. 
It follows from Lemma 2 in \citep{minsker2020asymptotic} (see also Theorem 3.1 in \citep{minsker2018uniform}) that under the stated assumptions, 
\[
\sup_{\theta\in \Theta} \l\lvert  \hL(\theta) - L(\theta,\theta') \r\rvert  \leq o(1) + C\wt \Delta \frac{\m O}{k\sqrt n}
\]
with probability at least $99\%$ as long as $n,k$ are large enough and $\m O/k$ is sufficiently small. Here, $o(1)$ is a function that tends to $0$ as $n\to \infty$. This shows consistency of $\wh\theta_N$. Next, we will provide the required explicit upper bound on $\|\wh\theta_N-\theta_0\|$. As we've demonstrated above, it suffices to find an upper bound for $L(\wh\theta_N) - L(\theta_0,\theta')$. We will apply the result of Theorem 2.1 in \citep{mathieu2021excess} to deduce that for $C$ large enough,
$L(\wh\theta_N) - L(\theta_0) \leq C\l( \wt \Delta \l( \frac{\m O+1}{k\sqrt n} + \l(\frac{k}{N}\r)^{\frac{1+\tau}{2}} \r) \r)+ O\l( \frac{1}{\sqrt N}\r)$ with probability at least $99\%$, where $\tau\in(0,1]$ is a constant from assumption \ref{ass:3B}. To see this, it suffices to notice that in view of Lemma \ref{lemma:sup-power} in the appendix, $\mb E\sup_{\theta\in \Theta}\l\lvert  \frac{1}{N}\sum_{j=1}^N \ell(\theta,X_j) -\ell(\theta',X_j) - L(\theta,\theta')\r\rvert  \leq \frac{C}{\sqrt N}$ where $C$ may depend on $\Theta$ and the class $\{p_\theta, \ \theta\in \Theta\}$. 


The proof of Part (2) of Theorem \ref{th:1} relies on many intermediary results established in the course of the proof of Theorem \ref{th:2}. Therefore, it is presented after the proof of Theorem \ref{th:2} in section \ref{proof:C} in the appendix.

\subsection{Proof of Theorem \ref{th:2} (sketch).}
\label{sec:sketch_of_th2}

The high-level idea of the proof is fairly standard and consists in obtaining a proper (quadratic) local approximation of $\wh L(\theta)$ in the neighborhood of $\theta_0$, coupled with careful control of the remainder terms. However, the difficulty that one has to overcome is the fact that, unlike the empirical log-likelihood, the robust estimator $\wh L(\theta)$ is not linear in $-\log p_\theta(X)$. To do so, we develop the technical tools that are based on the existing results in the papers  \citep{minsker2020asymptotic,minsker2018uniform}.

In view of the well-known property of the total variation distance,  
\ml{
\left\|\wh{\Pi}_N - \mathcal{N}\left(\wt \theta_N, \frac{1}{N}(\partial_\theta^2L(\theta_0))^{-1}\right)\right\|_{\mathrm{TV}} 
\\
 = \frac{1}{2}\int_\Theta \Bigg\lvert\frac{\pi(\theta)e^{-N\wh{L}(\theta)}}{\int_\Theta \pi(\theta')e^{-N\wh{L}(\theta')}d\theta'} 
 - \frac{N^{d/2}\lvert\partial_\theta^2L(\theta_0)\rvert^{1/2}}{(2\pi)^{d/2}}e^{-{\frac{1}{2}N(\theta - \wt \theta_N)^T\partial_\theta^2L(\theta_0)(\theta - \wt \theta_N)}} \Bigg\rvert d\theta.
}

Next, let us introduce the new variable $h = \sqrt{N}(\theta - \theta_0)$, multiply the numerator and the denominator on the posterior by $N\wh L(\theta_0)$, and set 
\mln{
\label{eq:kappa}
\kappa_N(h) = -N(\wh L(\theta_0 + h/\sqrt{N}) - \wh L(\theta_0)) 
\\
- \frac{N}{2}(\partial_\theta \wh L(\theta_0))^T(\partial_\theta^2L(\theta_0))^{-1}\partial_\theta\wh L(\theta_0),
} 
and $K_N = \int_{\mb R^d} \pi(\theta_0 + h/\sqrt{N})e^{\kappa_N(h)}d\mu(h)$. The total variation distance can then be equivalently written as
$\left\|\hat{\Pi}_N - \mathcal{N}\left(\wt \theta_N\,, \frac{1}{N}(\partial_\theta^2L(\theta_0))^{-1}\right)\right\|_{\mathrm{TV}} 
= \frac{1}{2}\int \Big\lvert\frac{\pi(\theta_0 + h/\sqrt{N})e^{\kappa_N(h)}}{K_N} - \frac{\lvert\partial_\theta^2L(\theta_0)\rvert^{1/2}}{(2\pi)^{d/2}}e^{-{\frac{1}{2}(h - \sqrt{N}(\wt \theta_N - \theta_0))^T\partial_\theta^2L(\theta_0)(h - \sqrt{N}(\wt \theta_N - \theta_0))}} \Big\rvert dh$. 
The function $\pi(\theta_0 + h/\sqrt{N})e^{\kappa_N(h)}/K_N$ can be viewed a probability density function of a new probability measure $\hat{\Pi}_N'$. Thus it suffices to show that
\begin{align*}
    \l\|\hat{\Pi}'_N - \mathcal{N}\l(\sqrt{N}(\wt \theta_N - \theta_0), (\partial_\theta^2L(\theta_0))^{-1}\r)\r\|_{TV} \overset{P}{\longrightarrow} 0\,.
\end{align*}
Since $\theta_0$ is the unique minimizer of $L(\theta)$, $\partial_\theta L(\theta_0) = 0$. Next, define
$H(\theta, z) = \sum_{j=1}^k \rho'\l(\sqrt{n}\frac{\Bar{L}_j(\theta) - z}{\Delta_n}\r)$; it is twice differentiable since both $\rho$ and $\ell$ are. It is shown in the proof of Lemma 4 in \citep{minsker2020asymptotic} that 
$\partial_z H(\theta,\wh L(\theta_0)) \neq 0$ with high probability. Therefore, a unique mapping $\theta \mapsto \wh L(\theta)$ exists around the neighborhood of $\theta_0$ and so do $\partial_\theta\wh L(\theta_0)$ and $\partial_\theta^2 \wh L(\theta_0)$. 
Denote 
\[
Z_N = -(\partial^2_\theta L(\theta_0))^{-1}\sqrt{N}\,\partial_\theta \wh L(\theta_0).
\]
The following result, proven in the appendix, essentially establishes stochastic differentiability of $\wh L(\theta)$ at $\theta=\theta_0$.  
\begin{lemma}
\label{lemma:technical1}
The following relation holds: 
\[
\sqrt{N}\l(\wt \theta_N - \theta_0\r) - Z_N \overset{P}{\longrightarrow} 0.
\]
\end{lemma}
In view of the lemma, the total variation distance between the normal laws $\mathcal{N}\l(\sqrt{N}(\wt \theta_N - \theta_0), \l(\partial^2_\theta L(\theta_0)\r)^{-1}\r)$ and $\mathcal{N}\l(Z_N,(\partial^2_\theta L(\theta_0))^{-1}\r)$ converges to $0$ in probability. Hence one only needs to show that $\l\|\hat{\Pi}'_N - \mathcal{N}\l(Z_N,(\partial^2_\theta L(\theta_0))^{-1}\r)\r\|_{TV} \overset{P}{\longrightarrow} 0$. 
Let
\begin{align}\label{eq:lamda}
    \lambda_N(h) = -\frac{1}{2}(h-Z_N)^T\partial^2_\theta L(\theta_0)(h-Z_N)
\end{align}
and observe that as long as one can establish that
\begin{equation}
\label{eq:equibvm}
\int_{\mb R^d}\l\lvert\pi(\theta_0 + h/\sqrt{N})e^{\kappa_N(h)} - \pi(\theta_0)e^{\lambda_N(h)}\r\rvert d h \overset{P}{\longrightarrow} 0\,,
\end{equation}
we will be able to conclude that 
\ml{
\l\lvert K_N - (2\pi)^{d/2}\lvert\partial^2_\theta L(\theta_0)\rvert^{-1}\pi(\theta_0)\r\rvert 
\\ = \l\lvert K_N - \int_{\mb R^d}\pi(\theta_0)e^{\lambda_N(h)}d\mu(h)\r\rvert 
\\
\leq \int_{\mb R^d}\l\lvert\pi(\theta_0 + h/\sqrt{N})e^{\kappa_N(h)} - \pi(\theta_0)e^{\lambda_N(h)}\r\rvert d h \overset{P}{\longrightarrow} 0,
} 
so that $K_N \overset{P}{\longrightarrow} (2\pi)^{d/2}\lvert \partial^2_\theta L(\theta_0)\rvert^{-1}\pi(\theta_0)$. This further implies that
\ml{
    \int_{\mathbb{R}^d}\l\lvert\pi(\theta_0 + h/\sqrt{N})e^{\kappa_N(h)} - \frac{K_N\lvert\partial^2_\theta L(\theta_0)\rvert}{(2\pi)^{d/2}}e^{\lambda_N(h)}\r\rvert d h \\
    \leq \int_{\mb R^d}\l\lvert\pi(\theta_0 + h/\sqrt{N})e^{\kappa_N(h)} - \pi(\theta_0)e^{\lambda_N(h)}\r\rvert d\mu(h)
    \\ 
    + \l\lvert \pi(\theta_0)- \frac{K_N \lvert \partial^2_\theta L(\theta_0)\rvert}{(2\pi)^{d/2}}\r\rvert\int_{\mb R^d} e^{\lambda_N(h)}d h 
    \\
    =  \l\lvert\pi(\theta_0) - \frac{K_N\lvert\partial^2_\theta L(\theta_0)\rvert}{(2\pi)^{d/2}}\r\rvert\frac{(2\pi)^{d/2}}{\lvert\partial^2_\theta L(\theta_0)\rvert} \\
    +\int_{\mb R^d} \l\lvert\pi(\theta_0 + h/\sqrt{N})e^{\kappa_N(h)} - \pi(\theta_0)e^{\lambda_N(h)}\r\rvert d h \overset{P}{\longrightarrow} 0,
}
and the desired result would follow. Therefore, it suffices to establish that relation \eqref{eq:equibvm} holds. Moreover, since $\pi = 0$ outside of a compact set $\Theta$, it is equivalent to showing that
\begin{equation}
\label{eq:integral}
\int_{\Theta'} \l\lvert \pi(\theta_0 + h/\sqrt{N})e^{\kappa_N(h)} - \pi(\theta_0)e^{\lambda_N(h)}\r\rvert d\mu(h) \overset{P}{\longrightarrow} 0,
\end{equation}
where $\Theta' = \{h : \theta_0 + h/\sqrt{N} \in \Theta\}$. Note that
\begin{align*}
\partial_\theta \wh L(\theta_0 + h/\sqrt{N}) - \partial_\theta \wh L(\theta_0) 
= \frac{1}{\sqrt{N}}\partial_\theta^2 \wh L(\theta_0)h + o_P(\|h\|/\sqrt{N})\,.
\end{align*}
An argument behind the proof of Lemma \ref{lemma:technical1} yields (again, we present the missing details in the technical supplement) the following representation for $\kappa_N(h)$ defined in \eqref{eq:kappa}: 
\begin{multline}
\label{eq:kappa_expansion}
    \kappa_N(h) = -\sqrt{N}h^T\partial_\theta \wh L(\theta_0) - \frac{1}{2}h^T\partial^2_\theta L(\theta_0)h 
    \\- \frac{N}{2}(\partial_\theta \wh L(\theta_0))^T(\partial_\theta^2L(\theta_0))^{-1}\partial_\theta\wh L(\theta_0) \\
    - N\l(R_1(\theta_0 + h/\sqrt{N}) + R_2(\theta_0 + h/\sqrt{N})\r)\,.
\end{multline}
Let us divide $\Theta'$ into 3 regions: $A_N^1 = \{h \in \Theta' : \|h\| \leq \|h^0_N\|\}$, $A_N^2 = \{h \in \Theta' : \|h^0_N\| < \|h\| \leq \delta\sqrt{N}\}$ and $A_N^3 = \{h \in \Theta' : \delta\sqrt{N}< \|h\| \leq R\sqrt{N}\}$ where $\delta$ is a sufficiently small positive number and $R$ is a sufficiently large so that $\{\theta \in \mb R^d : \|\theta - \theta_0\| \leq R\}$ contains $\Theta$. Finally, $h^0_N$ is chosen such that $\|h^0_N\|\rightarrow\infty$, $\|h^0_N/\sqrt{N}\| \rightarrow 0$ and that satisfies an additional growth condition specified in Lemma \ref{lemma:negligible_remainder} in the appendix. The remainder of the proof is technical and is devoted to proving that each part of the integral \eqref{eq:integral} corresponding to $A_N^1, \ A_N^2, \ A_N^3$ converges to $0$. Details are presented in the appendix.


\section*{Declarations}

\subsubsection*{Funding} 

Authors acknowledge support by the National Science Foundation grants DMS CAREER-2045068 and CIF-1908905.

\subsection*{Conflicts of Interests}

The authors have no relevant financial or non-financial interests to disclose.

\subsection*{Ethics approval}

No ethics approval is applicable to this article.

\subsection*{Consent to participate}

No consent to participate is applicable to this article.

\subsection*{Consent for publication}

No consent for publication is applicable to this article.

\subsection*{Availability of Data}

The data used in this article is available at UC Irvine machine learning repository:  \textcolor{blue}{https://archive.ics.uci.edu/dataset/186/wine+quality}.

\subsection*{Code Availability}

The code for this paper is stored in \textcolor{blue}{https://github.com/shunanyao/MOM-Bayes}.

\subsection*{Authors' Contribution}

The idea and concept was proposed by Stanislav Minsker. Data analysis was done by Shunan Yao. Mathematical proofs and the draft of the manuscript were written by both authors. Both authors read and approved the manuscript.

\bibliography{RobustPosterior}

\newpage
\appendix

\section{Technical results.}

In this section, we introduce some of the technical tools that will be used in the proofs of our main results. Lemmas \ref{lemma:unif} - \ref{lemma:deriv-norm} stated below were established in \cite{minsker2020asymptotic}, and therefore will be given without the proofs. Let $\Theta' \subset \Theta$ be a compact set, and define
\begin{align*}
    \Tilde{\Delta}(\Theta') := \max(\Delta_n, \sigma^2(\Theta'))\,.
\end{align*}
The following lemma provides a high probability bound for $\l\lvert\hL(\theta) - L(\theta)\r\rvert$ that holds uniformly over $\Theta' \subset \Theta$.
\begin{lemma}[Lemma 2 in \citep{minsker2020asymptotic}]
\label{lemma:unif}
Let $\m L = \{ \ell(\theta,\cdot), \ \theta \in \Theta \}$ be a class of functions mapping $S$ to $\mb R$, and assume that 
$\sup_{\theta\in \Theta} \mb E\l\lvert \ell(\theta,X) - L(\theta)\r\rvert^{2+\tau}<\infty$ for some $\tau\in [0,1]$. Then there exist absolute constants $c,\, C>0$ and a function $g_{\tau,P}(x)$ satisfying $g_{\tau,P}(x) \buildrel{x\to\infty}\over{=}\begin{cases} o(1), & \tau=0, 
\\
O(1), & \tau>0
\end{cases}$ such that for all $s>0,$ $n$ and $k$ satisfying
\be
\frac{s}{\sqrt{k}\Delta_n}\,\mb E\sup_{\theta\in \Theta'} \frac{1}{\sqrt{N}}\sum_{j=1}^N \l\lvert \ell(\theta,X_j)) - L(\theta) \r\rvert 
+ g_{\tau,P}(n)\sup_{\theta \in \Theta'} \frac{\mb E \l\lvert \ell(\theta,X) - L(\theta)\r\rvert^{2+\tau}}{\Delta_n^{2+\tau}n^{\tau/2}} \leq c,
\ee
the following inequality holds with probability at least $1 - \frac{1}{s}$:
\ml{
\sup_{\theta\in \Theta'}\l\lvert  
\hL(\theta) - L(\theta) \r\rvert  \leq 
C 
\Bigg[ s\cdot\frac{\widetilde \Delta(\Theta')}{\Delta_n} \mb E\sup_{\theta\in \Theta'} \l\lvert \frac{1}{N} \sum_{j=1}^N \Big( \ell(\theta,X_j) - L(\theta) \Big) \r\rvert  
\\
+ \widetilde \Delta(\Theta') \l( \frac{g_{\tau,P}(n)}{\sqrt n}\sup_{\theta \in \Theta'} \frac{\mb E \l\lvert  \ell(\theta,X) - L(\theta)\r\rvert ^{2+\tau}}{\Delta_n^{2+\tau}n^{\tau/2}} \r) \Bigg].
}
\end{lemma}
This lemma implies that as long as $\mb E \sup_{\theta \in \Theta'} N^{-1/2}\sum_{j=1}^N \l\lvert \ell(\theta, X_j) - L(\theta)\r\rvert  = O(1)$,
\begin{align*}
    \sup_{\theta \in \Theta'} \l\lvert \hL(\theta) - L(\theta)\r\rvert  = O_p(N^{-1/2} + n^{-(1 + \tau) / 2})\,.
\end{align*}

Next, Lemma \ref{lemma:consistency} below establishes consistency of the estimator $\wt \theta_N$, defined in \eqref{eq:theta-tilde}, that is a necessary ingredient in the proof of the Bernstein-von Mises theorem.
\begin{lemma}[Theorem 1 in \citep{minsker2020asymptotic}]
\label{lemma:consistency}
Let assumptions \ref{ass:1}-\ref{ass:3} hold, and assume that $\limsup_{N\to\infty}\frac{\m O}{k}\leq c$ for a sufficiently small constant $c>0$. Then $\wt{\theta}_{N} \rightarrow \theta_0$ in probability as $n,N/n\to\infty$. 
\end{lemma}

The following lemma establishes the asymptotic equicontinuity of the process $\theta\mapsto\partial_\theta \hL(\theta) - \partial_\theta L(\theta)$ at $\theta_0$ (recall that the existence of $\partial_\theta \hL(\theta)$ at the neighborhood around $\theta_0$ has bee justified in the proof of Theorem \ref{th:2}).

\begin{lemma}[Lemma 3 in \citep{minsker2020asymptotic}]
\label{lemma:equicont}
Let assumptions \ref{ass:1}-\ref{ass:3B} hold. Then for any $\eps>0$,
\[
\lim_{\delta\to 0}\limsup_{k,n\to\infty}\pr{\sup_{\|\theta - \theta_0\|\leq \delta}\l\| \sqrt{N}\l( \partial_\theta \hL(\theta) - \partial_\theta L(\theta) - \l(\partial_\theta \hL(\theta_0) - \partial_\theta L(\theta_0) \r)\r)\r\| \geq \eps} = 0.
\]
\end{lemma}

This result combined with Lemma \ref{lemma:consistency} yields a useful corollary.
Note that due to consistency of $\wt{\theta}_{N}$,
\begin{align}\label{eq:g1}
    \partial_\theta L(\wt{\theta}_{N}) - \partial_\theta L(\theta_0) = \partial_\theta^2 L(\theta_0)\l(\wt{\theta}_{N} - \theta_0\r) + o_p\l(\l\|\wt{\theta}_{N} - \theta_0\r\|\r)\,.
\end{align}
On the other hand,  
\mln{
\label{eq:g2}
\partial_\theta L(\wt{\theta}_{N}) - \partial_\theta L(\theta_0) = \partial_\theta \hL(\wt{\theta}_{N})
- \partial_\theta \hL(\theta_0) 
\\
+ \l( \partial_\theta L(\wt{\theta}_{N}) - \partial_\theta \hL(\wt{\theta}_{N}) \r) - \l( \partial_\theta L(\theta_0)- \partial_\theta \hL(\theta_0)\r) 
\\
=  \partial_\theta \hL(\wt{\theta}_{N}) - \partial_\theta \hL(\theta_0) + r_N,
}
where $r_N = \l( \partial_\theta L(\wt{\theta}_{N}) - \partial_\theta \hL(\wt{\theta}_{N}) \r) - \l( \partial_\theta L(\theta_0)- \partial_\theta \hL(\theta_0)\r)$. Note that for any $\delta>0$,
\begin{multline*}
\sqrt{N}\|r_N\| \leq \sqrt{N}\sup_{\|\theta - \theta_0\|\leq \delta} \l\| \l(\partial_\theta \hL(\theta) - \partial_\theta L(\theta)\r) - \l(\partial_\theta \hL(\theta_0) - \partial_\theta L(\theta_0)\r) \r\| \\
+ \sqrt{N}\|r_N\| I\{ \|\wt{\theta}_{N} - \theta_0 \| >\delta \}\,.
\end{multline*}
The first term converges to $0$ in probability by Lemma \ref{lemma:equicont} the second term converges to $0$ in probability by Lemma \ref{lemma:consistency}. Therefore,
\[
\pd^2_\theta L(\theta_0)\l(\wt{\theta}_{N} - \theta_0\r) + o\l(\|\wt{\theta}_{N} - \theta_0\|\r) = - \l( \partial_\theta \hL(\theta_0) - \partial_\theta L(\theta_0)\r) + o_p(N^{-1/2})\,.
\]
Under assumptions of the following Lemma \ref{lemma:deriv-norm},  $\sqrt{N}\l( \partial_\theta \hL(\theta_0) - \partial_\theta L(\theta_0)\r)$ is asymptotically (multivariate) normal, therefore, $\|\partial_\theta \hL(\theta_0) - \partial_\theta L(\theta_0)\| = O_p(N^{-1/2})$. Moreover, $\pd^2_\theta L(\theta_0)$ is non-singular by Assumption \ref{ass:2}. It follows that $\| \wt{\theta}_{N} - \theta_0 \| = O_p(N^{-1/2})$, and we conclude that 
\ben
\label{eq:decomposition}
\sqrt{N}(\wt{\theta}_{N} - \theta_0) = - \l( \pd^2_\theta L(\theta_0)\r)^{-1}\sqrt{N}\l( \partial_\theta \hL(\theta_0) - \partial_\theta L(\theta_0)\r) + o_p(1)\,.
\een 

\begin{lemma}[Lemma 4 and Theorem 3 in \citep{minsker2020asymptotic}]
\label{lemma:deriv-norm}
Let assumptions \ref{ass:1}-\ref{ass:3B} hold. Then the following asymptotic relations hold:
\begin{align*}
&\sqrt{N}\l( \partial_\theta \hL(\theta_0) - \partial_\theta L(\theta_0)\r) \overset{d}{\longrightarrow} \mathcal{N}\l(0, I(\theta_0) \r) \text{ and }
\\
&\sqrt{N}\l(\wt{\theta}_{N} - \theta_0\r) \overset{d}{\longrightarrow} \mathcal{N}\l(0, \l( \pd^2_\theta L(\theta_0)\r)^{-1} I(\theta_0) \l( \pd^2_\theta L(\theta_0)\r)^{-1}\r).
\end{align*}
\end{lemma}


The following lemma demonstrates that empirical processes indexed by classes that are Lipschitz in parameter (for example, satisfying assumption \ref{ass:3}) are ``well-behaved.'' This fact is well-known but we outline the proof for the reader's convenience.

\begin{lemma}
\label{lemma:sup-power}
Let $\m F = \l\{f_\theta, \ \theta\in \Theta'\subseteq \mb R^d\r\}$ be a class of functions that is Lipschitz in parameter, meaning that 
$\lvert f_{\theta_1}(x) - f_{\theta_2}(x)\rvert\leq M(x)\|\theta_1 - \theta_2\|$. Moreover, assume that $\mb EM^{2}(X)<\infty$ for some $p\geq 1$. Then 
\ml{
\mb E \sup_{\theta_1, \theta_2\in \Theta'} \frac{1}{\sqrt{n}} \l\lvert  \sum_{j=1}^{n}  \l(f_{\theta_1}(X_j) - f_{\theta_2}(X_j) - P(f_{\theta_1} - f_{\theta_2})\r)\r\rvert 
\\
\leq C \sqrt{d} \,\diam(\Theta',\|\cdot\|) \mb E \|M \|_{L_2(\Pi_n)}.
}
\end{lemma}
\begin{proof}
Let $\eps_1,\ldots,\eps_n$ be i.i.d. random signs that are independent from $X_1,\ldots,X_n$. Symmetrization inequality \cite{van1996weak} yields that 
\begin{multline*}
\mb E \sup_{\theta_1, \theta_2\in \Theta'} \frac{1}{\sqrt{n}} \l\lvert  \sum_{j=1}^{n}  \l(f_{\theta_1}(X_j) - f_{\theta_2}(X_j) - P(f_{\theta_1} - f_{\theta_2})\r)\r\rvert  
\\
\leq C\mb E \sup_{\theta_1, \theta_2\in \Theta'} \frac{1}{\sqrt{n}} \l\lvert  \sum_{j=1}^{n}  \eps_j\l(f_{\theta_1}(X_j) - f_{\theta_2}(X_j) \r)\r\rvert   
\\
= C\mb E_X \mb E_\eps \sup_{\theta_1, \theta_2\in \Theta'} \frac{1}{\sqrt{n}} \l\lvert  \sum_{j=1}^{n}  \eps_j\l(f_{\theta_1}(X_j) - f_{\theta_2}(X_j) \r)\r\rvert .
\end{multline*}
As the process $f\mapsto \frac{1}{\sqrt{n}}  \sum_{j=1}^{n}  \eps_j\l(f_{\theta_1}(X_j) - f_{\theta_2}(X_j) \r)$ is sub-Gaussian conditionally on $X_1,\ldots,X_n$, its (conditional) $L_p$-norms are equivalent to $L_1$ norm. Hence, Dudley's entropy bound implies that 
\begin{multline*}
\mb E_\eps \sup_{\theta_1, \theta_2\in \Theta'} \frac{1}{\sqrt{n}} \l\lvert  \sum_{j=1}^{n}  \eps_j\l(f_{\theta_1}(X_j) - f_{\theta_2}(X_j) \r)\r\rvert  
\\
\leq C  \mb E_\eps \sup_{\theta_1, \theta_2\in \Theta'} \frac{1}{\sqrt{n}} \l\lvert  \sum_{j=1}^{n}  \eps_j\l(f_{\theta_1}(X_j) - f_{\theta_2}(X_j) \r)\r\rvert  
\leq C \int_{0}^{D_n(\Theta')} H^{1/2}(z,T_n,d_n) dz,
\end{multline*}
where $d^2_n(f_{\theta_1},f_{\theta_2}) = \frac{1}{n}\sum_{j=1}^n \l( f_{\theta_1}(X_j) - f_{\theta_2}(X_j)\r)^2$, 
$T_n = \l\{ (f_{\theta}(X_1),\ldots, f_{\theta}(X_n)), \ \theta \in \Theta'\r\}\subseteq \mb R^n$ and $D_n(\Theta')$ is the diameter of $\Theta$ with respect to the distance $d_n$. 
As $f_{\theta}(\cdot)$ is Lipschitz in $\theta$, we have that 
$d^2_n(f_{\theta_1},f_{\theta_2}) \leq \frac{1}{n}\sum_{j=1}^n M^2(X_j) \|\theta_1 - \theta_2\|^2$, implying that $D_n(\Theta')\leq \|M\|_{L_2(\Pi_n)}\diam(\Theta',\|\cdot\|)$ and
\begin{equation}
H(z,T_n,d_n)\leq H\l(z/\| M \|_{L_2(\Pi_n)},\Theta',\|\cdot\|\r) 
\leq \log\l(C\frac{\diam(\Theta', \|\cdot\|)\, \|M \|_{L_2(\Pi_n)}}{z}\r)^d.
\end{equation}
Therefore, 
\[
\int_{0}^{D_n(\Theta')} H^{1/2}(z,T_n,d_n) dz \leq C \sqrt{d} \diam(\Theta',\|\cdot\|) \cdot \| M \|_{L_2(\Pi_n)}
\]
and 
\begin{equation}
\mb E_X \mb E_\eps \sup_{\theta_1, \theta_2\in \Theta'} \frac{1}{\sqrt{n}} \l\lvert  \sum_{j=1}^{n}  \eps_j\l(f_{\theta_1}(X_j) - f_{\theta_2}(X_j) \r)\r\rvert  
\leq C \sqrt{d} \diam(\Theta',\|\cdot\|) \mb E^{1/2} \|M \|^2_{L_2(\Pi_n)}.
\end{equation}
\end{proof}

\noindent Next is a lemma needed for the proof of the second part of Theorem \ref{th:1}. Let $\wh L(\theta)'$ be the ``outlier-free'' version of $\wh L(\theta)$.
\begin{lemma}
\label{lemma:outlier-effect}
    Let $\wh L'(\theta_0)$ be defined by \eqref{eq:lhat_no_outliers}. Then,
    \begin{align*}
        \l\lvert \wh L(\theta) - \wh L'(\theta) \r\rvert \leq \frac{C(\rho)\mathcal{O}\Delta_n}{k\sqrt{n}}\,,
    \end{align*}
uniformly over $\Theta$ with probability converging to $1$ as $N\to\infty$. Here, $C(\rho)$ is a universal constant only related to $\rho$.
\end{lemma}
\begin{proof}
    Note that
    \begin{multline*}
        0 = \sum_{j=1}^k\rho'\l(\sqrt{n}\frac{\Bar{L}_j(\theta) - \wh L(\theta)}{\Delta_n}\r) - \sum_{j=1}^k\rho'\l(\sqrt{n}\frac{\Bar{L}'_j(\theta) - \wh L'(\theta)}{\Delta_n}\r) \\
        = \l(\sum_{j \in \mathcal{I}}\rho'\l(\sqrt{n}\frac{\Bar{L}_j(\theta) - \wh L(\theta)}{\Delta_n}\r) - \sum_{j \in \mathcal{I}}\rho'\l(\sqrt{n}\frac{\Bar{L}'_j(\theta) - \wh L'(\theta)}{\Delta_n}\r)\r) \\
        + \l(\sum_{j \in [k]\backslash \mathcal{I}}\rho'\l(\sqrt{n}\frac{\Bar{L}_j(\theta) - \wh L(\theta)}{\Delta_n}\r) - \sum_{j \in [k]\backslash \mathcal{I}}\rho'\l(\sqrt{n}\frac{\Bar{L}'_j(\theta) - \wh L(\theta)'}{\Delta_n}\r)\r)\,.
    \end{multline*}
    Here, $\mathcal{I} = \{j: \Bar{L}_j'(\theta) = \Bar{L}_j(\theta)\}$, i.e., the set of indices of blocks that do not contain outliers. Since the number of outliers is $\mathcal{O}$, then the cardinality of $[k]\backslash \mathcal{I}$ is at most $\mathcal{O}$. Observe that $\l\lvert \sum_{j \in [k]\backslash \mathcal{I}}\rho'\l(\sqrt{n}\frac{\Bar{L}_j(\theta) - \wh L(\theta)}{\Delta_n}\r) - \sum_{j \in [k]\backslash \mathcal{I}}\rho'\l(\sqrt{n}\frac{\Bar{L}'_j(\theta) - \wh L'(\theta)}{\Delta_n}\r)\r\rvert \leq \mathcal{O}\|\rho'\|$. 
    Consequently,
    \begin{multline}
    \label{eq:bounded_rho'}
        \l\lvert \sum_{j \in \mathcal{I}}\rho'\l(\sqrt{n}\frac{\Bar{L}_j(\theta) - \wh L(\theta)}{\Delta_n}\r) - \sum_{j \in \mathcal{I}}\rho'\l(\sqrt{n}\frac{\Bar{L}'_j(\theta) - \wh L'(\theta)}{\Delta_n}\r) \r\rvert \\
        = \l\lvert \sum_{j \in \mathcal{I}}\rho'\l(\sqrt{n}\frac{\Bar{L}_j(\theta) - \wh L(\theta)}{\Delta_n}\r) - \sum_{j \in \mathcal{I}}\rho'\l(\sqrt{n}\frac{\Bar{L}_j(\theta) - \wh L'(\theta)}{\Delta_n}\r) \r\rvert \leq \mathcal{O}\|\rho'\|\,. 
    \end{multline}
    This implies that
    \begin{align*}
        \max \l(\l\lvert \sqrt{n}\frac{\Bar{L}'_j(\theta) - \wh L(\theta)}{\Delta_n}\r\rvert, \l\lvert\sqrt{n}\frac{\Bar{L}'_j(\theta) - \wh L'(\theta)}{\Delta_n}\r\rvert\r) \leq 1\,,
    \end{align*}
    for $j \in J \subset \mathcal{I}$ where $\lvert J \rvert \leq \frac{k}{M_N}$. 
    To see this, note that since we assume $\l\lvert \wh L'(\theta) - \wh L(\theta) \r\rvert \geq \frac{2\|\rho'\|\mathcal{O}\Delta_n}{k\sqrt{n}}$, if $\lvert J \rvert > \frac{k}{M_N}$,
    \begin{multline*}
        \l\lvert \sum_{j \in \mathcal{I}}\rho'\l(\sqrt{n}\frac{\Bar{L}'_j(\theta) - \wh L(\theta)}{\Delta_n}\r) - \sum_{j \in \mathcal{I}}\rho'\l(\sqrt{n}\frac{\Bar{L}'_j(\theta) - \wh L'(\theta)}{\Delta_n}\r) \r\rvert \\
        \geq \l\lvert \sum_{j \in J}\rho'\l(\sqrt{n}\frac{\Bar{L}'_j(\theta) - \wh L(\theta)}{\Delta_n}\r) - \sum_{j \in J}\rho'\l(\sqrt{n}\frac{\Bar{L}'_j(\theta) - \wh L'(\theta)}{\Delta_n}\r) \r\rvert \\
        = \frac{\sqrt{n}}{\Delta_n}\lvert J\rvert  \l\lvert \wh L'(\theta) - \wh L(\theta) \r\rvert > \mathcal{O}\|\rho'\|\,,
    \end{multline*}
    where the equality is given by the fact that $\rho'(z) = z$ for $\lvert z \rvert \leq 1$ in Assumption \ref{ass:1}. This contradicts with the previous conclusion in \eqref{eq:bounded_rho'}. Since $\sup_{\theta \in \Theta} \sqrt{n} \lvert \wh L(\theta) - L(\theta) \rvert$ and $\sup_{\theta \in \Theta} \sqrt{n} \lvert \wh L'(\theta) - L(\theta) \rvert$ both converge to $0$ by Lemma \ref{lemma:unif}, it suffices to show that
    \begin{align*}
        \mb P \l(\exists\theta, \l\lvert \sqrt{n}\frac{\Bar{L}'_j(\theta) - L(\theta)}{\Delta_n} \r\rvert \leq \frac{1}{2} \text{ for } j \in J' \text{ where } \lvert J' \rvert \leq \frac{k}{M_N} + \mathcal{O}\r)
    \end{align*}
    converges to $0$. Note that this probability is bounded by
    \begin{align*}
        \mb P \l(\sup_\theta \l\lvert \sqrt{n}\frac{\Bar{L}'_j(\theta) - L(\theta)}{\Delta_n} \r\rvert \leq \frac{1}{2} \text{ for } j \in J' \text{ where } \lvert J' \rvert \leq \frac{k}{M_N} + \mathcal{O}\r)\,.
    \end{align*}
    Since $\Theta$ is a compact set, there exist $\theta_1, \theta_2, \ldots, \theta_R$ such that $\bigcup_{j=1}^R B(\theta_j, \delta(\theta_j)) \supset \Theta$ where $\delta(\theta_j)$ is defined in assumption \ref{ass:3}. By Lemma \ref{lemma:sup-power},
    \begin{multline*}
        \mb P \l( \sup_{\theta} \l\lvert \sqrt{n}\frac{\l(\Bar{L}'_j(\theta) - L(\theta)\r) - \l(\Bar{L}'_j(\theta_0) - L(\theta_0)\r)}{\Delta_n} \r\rvert \geq \frac{1}{4} \r) \\
        \leq \sum_{j=1}^R \mb P\l(\sup_{\theta \in B(\theta_j, \delta(\theta_j))} \l\lvert \sqrt{n}\frac{\l(\Bar{L}'_j(\theta) - L(\theta)\r) - \l(\Bar{L}'_j(\theta_0) - L(\theta_0)\r)}{\Delta_n} \r\rvert \geq \frac{1}{4}\r) \\
        \leq 4 \sum_{j=1}^R \mb E \sup_{\theta \in B(\theta_j, \delta(\theta_j))} \l\lvert \sqrt{n}\frac{\l(\Bar{L}'_j(\theta) - L(\theta)\r) - \l(\Bar{L}'_j(\theta_0) - L(\theta_0)\r)}{\Delta_n} \r\rvert \\
        \leq C\Delta_n^{-1}\sum_{j=1}^R \mb E \|\m V(X, \delta(\theta_j))\|_{L_2(\Pi_n)} \rightarrow 0\,,
    \end{multline*}
    since $\Delta_n \rightarrow \infty$. Therefore,
    \begin{multline*}
        \mb P\l(\sup_\theta \l\lvert \sqrt{n}\frac{\Bar{L}'_j(\theta) - L(\theta)}{\Delta_n} \r\rvert \leq \frac{1}{2}\r) \geq
        \mb P\l(\l\lvert \sqrt{n}\frac{\Bar{L}'_j(\theta_0) - L(\theta_0)}{\Delta_n} \r\rvert \leq \frac{1}{4}\r) \\ -  \mb P \l( \sup_{\theta} \l\lvert \sqrt{n}\frac{\l(\Bar{L}'_j(\theta) - L(\theta)\r) - \l(\Bar{L}'_j(\theta_0) - L(\theta_0)\r)}{\Delta_n} \r\rvert \geq \frac{1}{4} \r)\,,
    \end{multline*}
    which is lower bounded for sufficiently large $N$ and small $\varepsilon$. Therefore, by Hoeffding's inequality,
    \begin{multline*}
        \mb P \l(\sup_\theta \l\lvert \sqrt{n}\frac{\Bar{L}'_j(\theta) - L(\theta)}{\Delta_n} \r\rvert \leq \frac{1}{2} \text{ for } j \in J' \text{ where } \lvert J' \rvert \leq \frac{k}{M_N} + \mathcal{O}\r) \\
        \leq \exp{-2k\l(\mb P \l(\sup_\theta \l\lvert \sqrt{n}\frac{\Bar{L}'_j(\theta) - L(\theta)}{\Delta_n} \r\rvert \leq \frac{1}{2}\r) - \l(\frac{1}{M_N} + \frac{\mathcal{O}}{k}\r)\r)^2} \rightarrow 0\,,
    \end{multline*}
    as $k \rightarrow \infty$ as long as $\mathcal{O} / k \rightarrow 0$.
\end{proof}

\noindent The next are three lemmas that we rely on in the proof of Theorem \ref{th:2}. Define 
\[
r_N(\theta) = \l( \partial_\theta L(\theta) - \partial_\theta \hL(\theta) \r) - \l( \partial_\theta L(\theta_0)- \partial_\theta \hL(\theta_0)\r).
\]

\begin{lemma}\label{lemma:taylor_expansion}
For any $\theta\in \Theta$,
\begin{align*}
    \wh L(\theta) = \wh L(\theta_0) + (\theta - \theta_0)^T \partial_\theta \wh L(\theta_0) + \frac{1}{2}(\theta - \theta_0)^T \partial^2_\theta L(\theta_0)(\theta - \theta_0) + R_1(\theta) + R_2(\theta)\,,
\end{align*}
where $R_1(\theta)$ and $R_2(\theta)$ are two functions such that for any $\theta$ satisfying $\|\theta - \theta_0\| \leq \delta$,
\[
R_1(\theta) \leq \|\theta - \theta_0\| \sup_{\theta}\|r_N(\theta)\|\,,
\quad \text{and} \quad
    \sup_{\theta}\l\lvert \frac{R_2(\theta)}{\|\theta - \theta_0\|^2}\r\rvert  \rightarrow 0\,,
\]
as $\delta \rightarrow 0$. 
\end{lemma}

Let $G_\theta(t) = \wh L\l(\theta_0 + t v\r)$ where $v = \theta - \theta_0$ and note that $G'_\theta(t) = v^T\partial_\theta \wh L(\theta_0 + tv)$. Then
\begin{equation}
\label{eq:f1}
    \wh L(\theta) = \wh L(\theta_0) + \int_0^1  G'_\theta(s)ds = \wh L(\theta_0) +\int_0^1 G'_\theta(0)ds +  \int_0^1 \l( G'_\theta(s) - G'_\theta(0)\r) ds\,.
\end{equation}
The first integral equals $\partial_\theta \wh L(\theta_0)(\theta - \theta_0)$. For the second integral, note that the reasoning similar to the one behind equation \eqref{eq:g1} yields that for any $\theta'$
\begin{align*}
    \partial_\theta \wh L(\theta') - \partial_\theta \wh L(\theta_0) - r_N(\theta') = \partial_\theta^2 L(\theta_0)(\theta' - \theta_0) + R(\theta' - \theta_0)\,,
\end{align*}
where $R(\theta - \theta_0)$ is a vector-valued function such that $R(\theta - \theta_0) / \|\theta - \theta_0\| \rightarrow 0$ as $\theta \rightarrow \theta_0$. Therefore, for any $s\in(0,1)$
\begin{multline*}
    G'_\theta(s) - G'_\theta(0) = v^T\l(\partial_\theta \wh L(\theta_0 + sv) - \partial_\theta \wh L(\theta_0)\r) \\
    = s \,v^T\partial_\theta^2 L(\theta_0)v + v^T\,r_N(\theta_0 + sv) + v^TR(sv),
\end{multline*}
implying that
\begin{multline*}
    \int_0^1 \left( G'_\theta(s) - G'_\theta(0) \right)ds = \int_0^1  \left( s\,v^T\partial_\theta^2 L(\theta_0)v + v^T r_N(\theta_0 + sv) + v^T R(sv) \right) ds 
    \\
    = \frac{1}{2} v^T\partial_\theta^2 L(\theta_0)v + \int_0^1v^Tr_N(\theta_0 +sv)ds + \int_0^1 v^TR(sv)ds\,.
\end{multline*}
Denoting the last two terms $R_1(\theta)$ and $R_2(\theta)$ respectively, and combining the previous display with equation \eqref{eq:f1}, we deduce that 
\begin{align*}
    \wh L(\theta) = \wh L(\theta_0) + \partial_\theta \wh L(\theta_0)(\theta - \theta_0) + \frac{1}{2}(\theta - \theta_0)^T \partial^2_\theta L(\theta_0)(\theta - \theta_0) + R_1(\theta) + R_2(\theta).
\end{align*}
Moreover,
\begin{align*}
    \int_0^1v^Tr_N(\theta_0 +sv)ds \leq \int_0^1 \|v\|\sup_{t \in [0,1]}\|r_N(\theta_0 + tv)\|ds = \|v\|\sup_{t \in [0,1]}\|r_N(\theta_0 + tv)\|\,.
\end{align*}
Therefore, for $\delta>0$ such that $\|\theta - \theta_0\| \leq \delta$, $R_1(\theta) \leq \|\theta - \theta_0\| \sup_{\|\theta - \theta_0\| \leq \delta}\|r_N(\theta)\|\,.$
Furthermore, 
\[
\int_0^1 v^TR(sv)ds \leq \int_0^1 \|v\|\sup_{t \in [0,1]}\|R(tv)\|ds = \|v\|\sup_{t \in [0,1]}\|R(tv)\|\,.
\]
and 
\begin{align*}
    \frac{R_2(\theta_0 + tv)}{\|v\|^2} \leq \frac{\sup_{t \in [0,1]}\|R(tv)\|}{\|v\|} \leq \sup_{t \in [0,1]}\l\|\frac{R(tv)}{\|tv\|}\r\|\,,
\end{align*}
Thus, $\sup_{\|\theta - \theta_0\| \leq \delta}\l\lvert \|R_2(\theta)\|/\|\theta - \theta_0\|^2\r\rvert  \leq \sup_{\|\theta - \theta_0\| \leq \delta}\l\|R(\theta - \theta_0)\r\|/\|\theta - \theta_0\|$ which converges to $0$ as $\delta \rightarrow 0$.

\begin{lemma}
\label{lemma:negligible_remainder}
There exits a sequence $h^0_N$ such that $\|h^0_N\|\rightarrow\infty$, $\|h^0_N/\sqrt{N}\| \rightarrow 0$ and 
\begin{align*}
    \sup_{h:\|h\|\leq\|h^0_N\|}\l\lvert N\l(R_1(\theta_0 + h/\sqrt{N}) +R_2(\theta_0 + h/\sqrt{N}) \r)\r\rvert \overset{P}{\longrightarrow} 0\,,
\end{align*}
where $R_1$ and $R_2$ are defined in Lemma \ref{lemma:taylor_expansion}.
\end{lemma}
\begin{proof}[Proof of Lemma \ref{lemma:negligible_remainder}.]
Let $h^*_N$ be a sequence such that $\|h^*_N\| \rightarrow \infty$ and $\|h^*_N/\sqrt{N}\| \rightarrow 0$. In view of Lemma \ref{lemma:equicont},
\begin{align*}
    \sup_{h:\|h\| \leq \|h^*_N\|}\|\sqrt{N}r_N(\theta_0 + h/\sqrt{N})\| \overset{P}{\longrightarrow} 0\,,
\end{align*}
where $r_N$ is given in Lemma \ref{lemma:taylor_expansion}. Moreover, let $h^{(1)}_N$ be a sequence such that $\|h^{(1)}_N\| \rightarrow \infty$, $\|h^{(1)}_N\| \leq \|h^*_N\|$ and
\begin{align*}
    \|h^{(1)}_N\|\sup_{h:\|h\| \leq \|h^*_N\|}\|\sqrt{N}r_N(\theta_0 + h/\sqrt{N})\| \overset{P}{\longrightarrow} 0\,.
\end{align*}
Lemma \ref{lemma:taylor_expansion} implies that
\begin{align*}
    \sup_{h : \|h\| \leq \|h^{(1)}_N\|}\l\lvert N R_1(\theta_0 + h/\sqrt{N})\r\rvert  &\leq \sup_{h : \|h\| \leq \|h^{(1)}_N\|}\l\lvert \sqrt{N} \|h\| \sup_{h : \|h\| \leq \|h^{(1)}_N\|} \|r_N(\theta_0 + h/\sqrt{N})\|\r\rvert  \\
    & \leq \|h^1_N\|\sup_{h : \|h\| \leq \|h^{(1)}_N\|} \|\sqrt{N}r_N(\theta_0 + h/\sqrt{N})\|\overset{P}{\longrightarrow} 0\,.
\end{align*}
Similarly, let $h^{**}_N$ be a sequence such that $\|h^{**}_N\| \rightarrow \infty$ and $\|h^{**}_N/\sqrt{N}\| \rightarrow 0$. Lemma \ref{lemma:taylor_expansion} yields that
\begin{align*}
    \sup_{h:\|h\|\leq\|h^{**}_N\|}\l\lvert \frac{R_2(\theta_0 + h/\sqrt{N})}{\|h/\sqrt{N}\|^2}\r\rvert \rightarrow 0\,.
\end{align*}
Finally, let $h^{(2)}_N$ be the sequence such that $h^{(2)}_N \rightarrow \infty$, $\|h^{(2)}_N\| \leq \|h^{**}_N\|$ and
\begin{align*}
    \l\|h^{(2)}_N\r\|^2 \sup_{h:\|h\|\leq\|h^{**}_N\|}\l\lvert \frac{R_2(\theta_0 + h/\sqrt{N})}{\|h/\sqrt{N}\|^2}\r\rvert  \rightarrow 0\,.
\end{align*}
Then
\begin{align*}
    \sup_{h:\|h\|\leq \|h^{(2)}_N\|}\l\lvert NR_2(\theta_0 + h/\sqrt{N})\r\rvert  &\leq \|h_N^{(2)}\|^2\sup_{h:\|h\|\leq \|h^{(2)}_N\|}\l\lvert \frac{R_2(\theta_0 + h/\sqrt{N})}{\|h/\sqrt{N}\|^2}\r\rvert  
    \\
    &\leq \|h^{(2)}_N\|^2\sup_{h:\|h\|\leq \|h^{**}_N\|}\l\lvert \frac{R_2(\theta_0 + h/\sqrt{N})}{\|h/\sqrt{N}\|^2}\r\rvert  \rightarrow 0\,.
\end{align*}
Finally, take $h^0_N = \argmin_{h \in \{h^{(1)}_N,h^{(2)}_N\}}\|h\|$, and conclude using the triangle inequality.
\end{proof}

\begin{lemma}\label{lemma:normal_tail}
Let $\{U_n\}_{n \geq 1}$ be a sequence of random vectors that converges to $UZ$ weakly, where $Z$ is a standard random vector of dimension $d$ and $U$ is a $d \times d$ invertible matrix. Furthermore, let $V$ be a $d \times d$ symmetric positive definite matrix and $\{a_n\}_{n\geq 1}$ - a sequence of positive numbers converging to infinity. Then
\begin{align*}
    \int_{\|h\| \geq a_n}e^{-\frac{1}{2}(h-U_n)^TV(h-U_n)}d\mu(h) \overset{P}{\longrightarrow} 0\,.
\end{align*}
\end{lemma}
\begin{proof}[Proof of Lemma \ref{lemma:normal_tail}.]
Note that
\begin{align*}
    \int_{\|h\| \geq a_n}e^{-\frac{1}{2}(h-U_n)^TV(h-U_n)}dh \leq \int_{\|h\| \geq a_n}e^{-\frac{1}{2}\lambda_{\min}^V\|h-U_n\|^2}dh\,,
\end{align*}
where $\lambda_{\min}^V$ is the smallest eigenvalue of $V$. Let $C$ be an arbitrary positive constant and $B_n = \{\|U_n\| \leq C\}$, then on the set $B_n$,
\begin{align*}
    \int_{\|h\| \geq a_n}e^{-\frac{1}{2}\lambda_{\min}^V\|h-U_n\|^2}dh \leq
    \int_{\|h\| \geq a_n}e^{-\frac{1}{2}\lambda_{\min}^V(\|h\| - C)^2}dh \leq \delta
\end{align*}
for any $\delta > 0$ as $n \rightarrow \infty$. Note that $\mathbb{P}(B_n) \rightarrow \mathbb{P}(\|UZ\| \leq C) \geq \mathbb{P}(\|Z\|^2 \leq C(\lambda_{\min}^{U^TU})^{-1})$, where $\lambda_{\min}^{U^TU}$ is the smallest eigenvalue of $U^TU$. For an arbitrary $\varepsilon>0$, select $C$ such that $\mathbb{P}(\|Z\|^2 \leq C(\lambda_{\min}^{U^TU})^{-1}) \geq 1 - \varepsilon$. Then
\begin{align*}
    \int_{\|h\| \geq a_n}e^{-\frac{1}{2}(h-U_n)^TV(h-U_n)}dh \leq \delta
\end{align*}
with probability at least $1 - \varepsilon$, thus the assertion holds.
\end{proof}

\section{Details of the proof of Theorem \ref{th:2}.}
\label{proof:B}

We begin by filling in the parts omitted in the sketch given in Section \ref{sec:sketch_of_th2}. First, Lemma \ref{lemma:technical1} is implied directly by display (\ref{eq:decomposition}), while display \eqref{eq:kappa_expansion} follows in view of Lemma \ref{lemma:taylor_expansion}. 
Recall that 
\begin{align*}
A_N^1 &= \{h \in \Theta' : \|h\| \leq \|h^0_N\|\}, 
\\
A_N^2 &= \{h \in \Theta' : \|h^0_N\| < \|h\| \leq \delta\sqrt{N}\}, 
\\
A_N^3 &= \{h \in \Theta' : \delta\sqrt{N}< \|h\| \leq R\sqrt{N}\}
\end{align*} 
where $\delta$ is a sufficiently small positive number and $R$ is a sufficiently large so that $\{\theta \in \mb R^d : \|\theta - \theta_0\| \leq R\}$ contains $\Theta$ (see the discussion following \eqref{eq:kappa_expansion} for the additional details). For the integral over the set $A_N^1$, observe that
\begin{multline*}
    \int_{A_N^1}\l\lvert \pi(\theta_0 + h/\sqrt{N})e^{\kappa_N(h)} - \pi(\theta_0)e^{\lambda_N(h)}\r\rvert d\mu(h) \leq 
    \\
    \int_{A_N^1}\pi(\theta_0 + h/\sqrt{N})\l\lvert e^{\kappa_N(h)}-e^{\lambda_N(h)}\r\rvert d\mu(h) 
    \\
    + \int_{A_N^1}\l\lvert \pi(\theta_0 + h/\sqrt{N}) - \pi(\theta_0)\r\rvert e^{\lambda_N(h)}d\mu(h)\,.
\end{multline*}
To estimate the first term in the right-hand side of the display above, recall the definition (\ref{eq:lamda}) of $\lambda_N$ and note that
\begin{align*}
    \lambda_N(h) = -\sqrt{N}h^T\partial_\theta \hL(\theta_0) - \frac{1}{2}h^T\partial^2_\theta L(\theta_0)h - \frac{N}{2}(\partial_\theta \hL(\theta_0))^T(\partial_\theta^2L(\theta_0))^{-1}\partial_\theta\hL(\theta_0)\,.
\end{align*}
Therefore, recalling that $\kappa_N$ can be written as in display (\ref{eq:kappa_expansion}), we have that
\begin{align*}
    \kappa_N(h) = \lambda_N(h)- N\l(R_1(\theta_0 + h/\sqrt{N}) + R_2(\theta_0 + h/\sqrt{N})\r),
\end{align*}
hence
\begin{multline*}
    \int_{A_N^1}\pi(\theta_0 + h/\sqrt{N})\l\lvert e^{\kappa_N(h)}-e^{\lambda_N(h)}\r\rvert dh \\ \leq \sup_{h \in A_N^1}\l\{\pi(\theta_0 + h/\sqrt{N})\l\lvert e^{-N\l(R_1(\theta_0 + h/\sqrt{N}) + R_2(\theta_0 + h/\sqrt{N})\r) } - 1\r\rvert \r\}\int_{h\in A_N^1}e^{\lambda_N(h)}d\mu(h)\,.
\end{multline*}
Here, $\sup_{h \in A_N^1}\pi(\theta_0 + h/\sqrt{N}) \rightarrow \pi(\theta_0)$ by the continuity of $\pi$ while 
\begin{align*}
    \sup_{h \in A_N^1}\l\lvert e^{- N\l(R_1(\theta_0 + h/\sqrt{N}) + R_2(\theta_0 + h/\sqrt{N})\r)} - 1\r\rvert  \overset{P}{\longrightarrow} 0\
\end{align*}
by Lemma \ref{lemma:negligible_remainder}. Moreover, by the definition of $\lambda_N$ (see equation \eqref{eq:lamda}), the integral factor equals $(2\pi)^{d/2}/\lvert\partial_\theta^2 L(\theta_0)\rvert$. Therefore, the first integral converges to $0$ in probability. For the second integral, observe that
\begin{multline*}
    \int_{A_N^1}\l\lvert \pi(\theta_0 + h/\sqrt{N}) - \pi(\theta_0)\r\rvert e^{\lambda_N(h)}dh 
    \\
    \leq \sup_{h \in A_N^1}\l\lvert \pi(\theta_0 + h/\sqrt{N}) - \pi(\theta_0)\r\rvert \int_{\mathbb{R}^d}e^{\lambda_N(h)}dh 
    \\
    = \sup_{h \in A_N^1}\l\lvert \pi(\theta_0 + h/\sqrt{N}) - \pi(\theta_0)\r\rvert \frac{(2\pi)^{d/2}}{\lvert\partial_\theta^2L(\theta_0)\rvert} \rightarrow 0\,,
\end{multline*}
by Assumption \ref{ass:5}. Next, to estimate the integral over $A_N^2$, note that
\begin{multline*}
    \int_{A_N^2}\l\lvert \pi(\theta_0 + h/\sqrt{N})e^{\kappa_N(h)} - \pi(\theta_0)e^{\lambda_N(h)}\r\rvert dh
    \\
    \leq \int_{A_N^2}\l\lvert \pi(\theta_0 + h/\sqrt{N})e^{\kappa_N(h)}\r\rvert dh 
    \\
    + \int_{A_N^2}\l\lvert \pi(\theta_0)e^{\lambda_N(h)}\r\rvert dh\,.
\end{multline*}
For the first term, consider again the representation of $\kappa_N$ as
\begin{multline*}
    \kappa_N(h) = -\sqrt{N}h^T\partial_\theta \hL(\theta_0) - \frac{1}{2}h^T\partial^2_\theta L(\theta_0)h - \frac{N}{2}(\partial_\theta \hL(\theta_0))^T(\partial_\theta^2L(\theta_0))^{-1}\partial_\theta\hL(\theta_0) \\
    - N\l(R_1(\theta_0 + h/\sqrt{N}) + R_2(\theta_0 + h/\sqrt{N})\r)\,.
\end{multline*}
Since $\partial_\theta^2 L(\theta_0)$ is a positive definite matrix, $\lambda_{\min}\l(\partial_\theta^2 L(\theta_0)\r) > 0$ and, in view of Lemma \ref{lemma:taylor_expansion},
\begin{multline*}
    \l\lvert N\l(R_1(\theta_0 + h/\sqrt{N}) + R_2(\theta_0 + h/\sqrt{N})\r)\r\rvert  \\
    \leq \|h\|^2\sup_{\|h\| \leq \delta\sqrt{N}}\l(\l\|\frac{\sqrt{N}r_N(\theta_0 + h/\sqrt{N})}{2\|h_0\|}\r\| + \frac{\lvert R_2(\theta_0 + h/\sqrt{N})\rvert}{\|h/\sqrt{N}\|^2}\r) \\ 
    \leq \frac{\lambda_{\min}\l(\partial_\theta^2 L(\theta_0)\r)}{4}\|h\|^2 \leq \frac{1}{4}h^T\partial_\theta^2L(\theta_0)h\,,
\end{multline*}
with probability close to $1$, for sufficiently small $\delta$. Then
\begin{align}\label{eq:normal_plus_W_N}
    \kappa_N(h) &\leq -\sqrt{N}h^T\partial_\theta \hL(\theta_0) - \frac{1}{4}h^T\partial^2_\theta L(\theta_0)h - \frac{N}{2}(\partial_\theta \hL(\theta_0))^T(\partial_\theta^2L(\theta_0))^{-1}\partial_\theta\hL(\theta_0) \\
    &= -\l(h - \frac{1}{2} Z_N\r)^T\partial_\theta^2 L(\theta_0)\l(h - \frac{1}{2}Z_N\r) + W_N\,,
\end{align}
where $W_N = \frac{1}{2} N(\partial_\theta \hL(\theta_0))^T(\partial_\theta^2L(\theta_0))^{-1}\partial_\theta\hL(\theta_0)$ and by Lemma \ref{lemma:deriv-norm},
\begin{align}\label{eq:W_N}
    W_N \rightarrow Z^THZ
\end{align}
weakly with $Z \sim \mathcal{N}\l(0,I_d\r)$ and $I_d$ and $H$ being a $d$-dimensional identity matrix $\frac{1}{2}I(\theta_0)(\partial_\theta^2L(\theta_0))^{-1}I(\theta_0)$ respectively. 
Therefore, for any positive increasing sequence $\{c_N\}$,
\begin{multline}\label{eq:A2_split}
    \int_{A_N^2}\l\lvert \pi(\theta_0 + h/\sqrt{N})e^{\kappa_N(h)}\r\rvert d\mu(h) \leq \\
    c_N\sup_{h \in A_N^2}\pi(\theta_0 + h/\sqrt{N})\int_{h \in A_N^2}e^{-\l(h - \frac{1}{2}Z_N\r)^T\partial_\theta^2L(\theta_0)\l(h - \frac{1}{2}Z_N\r)}d\mu(h) \\+ \sup_{h \in A_N^2}\pi(\theta_0 + h/\sqrt{N})e^{W_N}\int_{h \in A_N^2}e^{-\l(h - \frac{1}{2}Z_N\r)^T\partial_\theta^2L(\theta_0)\l(h - \frac{1}{2}Z_N\r)}d\mu(h) I\{W_N>\log c_N\}\,.
\end{multline}
It is easy to see that $\sup_{h \in A_N^2}\pi(\theta_0 + h/\sqrt{N})\int_{h \in A_N^2}e^{-\l(h - \frac{1}{2}Z_N\r)^T\partial_\theta^2L(\theta_0)\l(h - \frac{1}{2}Z_N\r)}d\mu(h)$ converges to $0$ in probability by Lemma \ref{lemma:normal_tail}. Then choosing $c_N$ such that
\begin{align*}
    c_N\sup_{h \in A_N^2}\pi(\theta_0 + h/\sqrt{N})\int_{h \in A_N^2}e^{-\l(h - \frac{1}{2}Z_N\r)^T\partial_\theta^2L(\theta_0)\l(h - \frac{1}{2}Z_N\r)}d\mu(h) \overset{P}{\longrightarrow} 0\,,
\end{align*}
guarantees that the first term converges to $0$. Meanwhile, the second term is $0$ with probability $\mathbb{P}\l(W_N \leq \log c_N\r)$. Note that for any $C$,
\begin{multline*}
    \mathbb{P}\l(W_N \leq \log c_N\r) = \mathbb{P}\l(W_N \leq \log c_N\r) - \mathbb{P}\l(W_N \leq \log C\r) \\
    + \mathbb{P}\l(W_N \leq \log C\r) - \mathbb{P}\l(Z^THZ \leq \log C\r) + \mathbb{P}\l(Z^THZ \leq \log C\r)\,.
\end{multline*}
$\mathbb{P}\l(W_N \leq \log c_N\r) - \mathbb{P}\l(W_N \leq \log C\r)$ is positive for $c_N$ large enough by tightness and $\mathbb{P}\l(W_N \leq \log C\r) - \mathbb{P}\l(Z^THZ \leq \log C\r)$ converges to $0$ by weak convergence. Thus, for $N$ large enough, $\mathbb{P}\l(W_N \leq \log c_N\r) \geq \mathbb{P}\l(Z^THZ \leq \log C\r)$ for any $C$. Since $I(\theta_0)$ and $\partial_\theta^2 L(\theta_0)$ are symmetric and positive definite, so is $H$. Note that for arbitrary $\varepsilon > 0$, one can select a sufficiently large $C$ such that
$\mathbb{P}\l(\|Z\|^2 > \frac{\log C}{\lambda_{\max} (H)}\r) \leq \varepsilon\,.$
Therefore,
\begin{align*}
  \mathbb{P}(Z^THZ > \log C) \leq \mathbb{P}(\lambda_{\max}(H)\|Z\|^2 > \log C) \leq \varepsilon\,.
\end{align*}
Thus, for $N$ large enough,
\begin{align*}
    \sup_{h \in A_N^2}\pi(\theta_0 + h/\sqrt{N})e^{W_N}\int_{h \in A_N^2}e^{-\l(h - \frac{1}{2}Z_N\r)^T\partial_\theta^2L(\theta_0)\l(h - \frac{1}{2}Z_N\r)}d\mu(h) I\{W_N>\log c_N\}
\end{align*}
equals $0$ with probability at least $1 - \varepsilon$ for any $\varepsilon$, hence the above term converges to $0$ in probability. We have so far shown that $\int_{A_N^2}\l\lvert \pi(\theta_0 + h/\sqrt{N})e^{\kappa_N(h)}\r\rvert d\mu(h)$ converges to $0$ in probability. Another application of Lemma \ref{lemma:normal_tail} implies that
\begin{align*}
    \int_{A_N^2}\l\lvert \pi(\theta_0)e^{\lambda_N(h)}\r\rvert d\mu(h) \leq \pi(\theta_0)\int_{\|h\| \geq a\log N}e^{\lambda_N(h)}d\mu(h) \overset{P}{\longrightarrow} 0\,,
\end{align*}
which shows the integral over $A_N^2$ converges to $0$ in probability. For the final part, the integral over $A_N^3$, observe again that
\begin{multline*}
    \int_{A_N^3}\l\lvert \pi(\theta_0 + h/\sqrt{N})e^{\kappa_N(h)} - \pi(\theta_0)e^{\lambda_N(h)}\r\rvert d\mu(h) \\ \leq \int_{A_N^3}\l\lvert \pi(\theta_0 + h/\sqrt{N})e^{\kappa_N(h)}\r\rvert d\mu(h) + \int_{A_N^3}\l\lvert \pi(\theta_0)e^{\lambda_N(h)}\r\rvert d\mu(h)\,.
\end{multline*}
As before, the second integral converges to $0$ in probability by Lemma \ref{lemma:normal_tail}. The first integral can be further estimated as
\begin{align*}
    \int_{A_N^3}\l\lvert \pi(\theta_0 + h/\sqrt{N})e^{\kappa_N(h)}\r\rvert d\mu(h)
    \leq \int_{a\leq\|h/\sqrt{N}\|\leq R}\l\lvert \pi(\theta_0 + h/\sqrt{N})e^{\kappa_N(h)}\r\rvert d\mu(h)\,.
\end{align*}
On the compact set $\{\theta: \delta \leq \|\theta - \theta_0\| \leq R\}$, $L(\theta) - L(\theta_0)$ attains a minimum positive value $t_1$. Moreover,
\begin{multline*}
    \inf_{a\leq\|h/\sqrt{N}\|\leq R}\hL(\theta_0 + h/\sqrt{N}) - \hL(\theta_0)
    \geq \inf_{\|h/\sqrt{N}\|\leq R}\l(\hL(\theta_0 + h/\sqrt{N}) - L(\theta_0 + h/\sqrt{N})\r) \\ + \inf_{a \leq \|h/\sqrt{N}\|\leq R}\l(L(\theta_0 + h/\sqrt{N}) - L(\theta_0)\r) + \l(L(\theta_0) - \hL(\theta_0)\r)\,.
\end{multline*}
By Lemma \ref{lemma:unif}, the terms in the first and third pair of brackets converge to $0$ in probability. Thus,
\begin{align}\label{eq:lower_bound_on_inf_of_ring}
    \inf_{a\leq\|h/\sqrt{N}\|\leq R}\hL(\theta_0 + h/\sqrt{N}) - \hL(\theta_0) \geq \frac{t_1}{2}\,,
\end{align}
with probability approaching 1. Therefore, by the definition of $\kappa_N(h)$ in \eqref{eq:kappa}, denoting $I_N = \inf_{a\leq\|h/\sqrt{N}\|\leq R}\hL(\theta_0 + h/\sqrt{N}) - \hL(\theta_0)$,
\begin{multline*}
    \int_{a\leq\|h/\sqrt{N}\|\leq R}\l\lvert \pi(\theta_0 + h/\sqrt{N})e^{\kappa_N(h)}\r\rvert d\mu(h) \\
    = \int_{a\leq\|h/\sqrt{N}\|\leq R}\l\lvert \pi(\theta_0 + h/\sqrt{N})e^{-NI_N - W_N}\r\rvert d\mu(h) \\
    \leq e^{-\frac{1}{2}Nt_1 - W_N}\int_{\mathbb{R}^d}\pi(\theta_0 + h/\sqrt{N})d\mu(h)
    \leq N^{d/2}e^{-\frac{1}{2}Nt_1 - W_N}\,,
\end{multline*}
where $W_N = \frac{1}{2} N(\partial_\theta \hL(\theta_0))^T(\partial_\theta^2L(\theta_0))^{-1}\partial_\theta\hL(\theta_0)$. Similar to the approach used in establishing the convergence in \eqref{eq:A2_split}, one can show $N^{d/2}e^{-\frac{1}{2}Nt_1 - W_N} \rightarrow 0$ with probability approaching 1. This establishes the relation (\ref{eq:equibvm}), and therefore completes the proof.

\section{Proof of Theorem \ref{th:1}, Part (2).}
\label{proof:C}


In this section, we establish the second claim of Theorem \ref{th:1}. Let $\Bar{L}'_j(\theta) = \frac{1}{n}\sum_{i \in G_j} \l(\ell(\theta, \tilde{X}_i) - \ell(\theta', \tilde{X}_i)\r)$ for every $j$ and
\begin{equation}\label{eq:lhat_no_outliers}
\wh L'(\theta) := \argmin_{z \in \mb{R}} \sum_{j=1}^k\rho\l(\sqrt{n}\frac{\Bar{L}'_j(\theta) - z}{\Delta_n}\r)\,.
\end{equation}
In particular, in the outlier-free framework $\mathcal{O} = 0$, $\wh L'(\theta) = \wh L(\theta)$. 
Define 
\[
\wh H(\theta_0):=(\partial_\theta \wh L'(\theta_0))^T (\partial_\theta^2 L(\theta_0))^{-1} \partial_\theta \wh L'(\theta_0),
\]
and note that the following indentity holds:
    \begin{multline*}
        \wh \Pi_N \l(\l\| \theta - \theta_0\r\| \geq \delta_N \mid \mathcal{X}_N \r) = \frac{\int_{\|\theta - \tilde{\theta}_N\| \geq \delta_N} \pi(\theta)e^{-N\wh L(\theta)} d\theta}{\int_\Theta \pi(\theta)e^{-N\wh L(\theta)} d\theta} 
        \\
        = \frac{\int\limits_{\|\theta - \theta_0\| \geq \delta_N} \pi(\theta)e^{-N\l(\wh L(\theta) - \wh L'(\theta)\r)}e^{-N\l(\wh L'(\theta) - \wh L'(\theta_0)\r) - \frac{N}{2}\wh H(\theta_0) } d\theta}{\int_\Theta \pi(\theta)e^{-N\l(\wh L(\theta) - \wh L'(\theta)\r)}e^{-N\l(\wh L'(\theta) - \wh L'(\theta_0)\r) - \frac{N}{2}\wh H(\theta_0)} d\theta} \,.
    \end{multline*}
    Since $\l\lvert \wh{L}(\theta) - \wh{L}'(\theta)\r\rvert \leq C(\rho)\Delta_n\frac{\mathcal{O}}{k\sqrt{n}}$ with probability approaching $1$ in view of Lemma \ref{lemma:outlier-effect}, the right hand side of the display above can be bounded as
    \begin{align}
    \label{eq:posterior_tails}
e^{2C(\rho)\mathcal{O}\Delta_n\sqrt{n}}\,\frac{\int\limits_{\|\theta - \theta_0\| \geq \delta_N} \pi(\theta)e^{-N\l(\wh L'(\theta) - \wh L'(\theta_0)\r) - \frac{N}{2}\wh H(\theta_0)} d\theta}{\int_\Theta \pi(\theta)e^{-N \l(\wh L'(\theta) - \wh L'(\theta_0)\r) - \frac{N}{2}\wh H(\theta_0)} d\theta}\,.
    \end{align}
    We will first show that the denominator in expression \eqref{eq:posterior_tails} is of order $O\l(\frac{1}{N^{d/2}}\r)$ while the numerator is of order $o\l(\frac{1}{N^{d/2}e^{\mathcal{O}\sqrt{n}\Delta_n}}\r)$. 
    The following relation was formally established in the course of the proof of Theorem \ref{th:2} (specifically, see display \eqref{eq:equibvm} and the following discussion): 
    \[
    \int_{\mb R^d}\l\lvert\pi(\theta_0 + h/\sqrt{N})e^{\kappa_N(h)} - \pi(\theta_0)e^{\lambda_N(h)}\r\rvert d h \overset{P}{\longrightarrow} 0\,.
    \]
    It implies that
    \begin{multline*}
        \int_\Theta \pi(\theta)e^{-N \l(\wh L'(\theta) - \wh L'(\theta_0)\r) - \frac{N}{2}\wh H(\theta_0)} d\theta 
        \\
        = \frac{1}{N^{d/2}} \int_{\Theta'} \pi\l(\theta_0 + \frac{h}{\sqrt{N}}\r)e^{-N \l(\wh L'(\theta_0 + h/\sqrt{N}) - \wh L'(\theta_0)\r) - \frac{N}{2}\wh H(\theta_0)} dh \\
        = \frac{1}{N^{d/2}} \l(\int_{\Theta'} \pi(\theta_0) e^{-\frac{1}{2}(h - Z_N)^T \partial_\theta^2 L(\theta_0) (h - Z_N)}dh + o_p(1)\r)\,,
    \end{multline*}
    where $Z_N = -(\partial_\theta^2 L(\theta_0))^{-1} \sqrt{N} \partial_\theta \wh L'(\theta_0)$ and $\Theta' = \{h : \theta_0 + h/\sqrt{N} \in \Theta\}$. Lemma \ref{lemma:deriv-norm} implies that the sequence $\{Z_N\}_{N\geq 1}$ converges weakly to the multivariate normal distribution, therefore, it is tight. Hence, the integral above is of order $O_p(1)$. Thus, one can conclude that the denominator in \eqref{eq:posterior_tails} is $O_p(1/N^{d/2})$ and the quantity in \eqref{eq:posterior_tails} is
    \begin{align*}
        O_p(N^{d/2})e^{2C(\rho)\mathcal{O}\Delta_n \sqrt{n}} \int_{\|\theta - \theta_0\| \geq \delta_N} \pi(\theta)e^{-N\l(\wh L'(\theta) - \wh L'(\theta_0)\r) - \frac{N}{2}\wh H(\theta_0)} d\theta
    \end{align*}
    Thus, it suffices to find the order of
    \begin{align*}
        \int_{\|\theta - \theta_0\| \geq \delta_N} \pi(\theta)e^{-N\l(\wh L'(\theta) - \wh L'(\theta_0)\r) - \frac{N}{2}\wh H(\theta_0)} d\theta\,.
    \end{align*}
    Similar to the proof of Theorem \ref{th:2}, the integral can be divided into two parts:
    \begin{align*}
        \int_{\delta_N \leq \|\theta - \theta_0\| \leq C_1} \pi(\theta)e^{-N\l(\wh L'(\theta) - \wh L'(\theta_0)\r) - \frac{N}{2}\wh H(\theta_0)} d\theta\,,
    \end{align*}
    and
    \begin{align*}
        \int_{C_1 \leq \|\theta - \theta_0\| \leq C_2} \pi(\theta)e^{-N\l(\wh L'(\theta) - \wh L'(\theta_0)\r) - \frac{N}{2}\wh H(\theta_0)} d\theta\,.
    \end{align*}
    Let us bound the second part first. Note that $\inf_{C_a \leq \lvert \theta - \theta_0 \rvert \leq C_2}\l(\wh L'(\theta) - \wh L'(\theta_0)\r) \geq t$ for some $t > 0$ in view of the inequality \eqref{eq:lower_bound_on_inf_of_ring} with probability converging to $1$ as long as $N$ is large enough. Therefore,
    \begin{multline*}
        \int_{C_1 \leq \|\theta - \theta_0\| \leq C_2} \pi(\theta)e^{-N\l(\wh L'(\theta) - \wh L'(\theta_0)\r)- \frac{N}{2}\wh H(\theta_0)} d\theta \\ 
        \leq e^{-Nt - \frac{N}{2}\wh H(\theta_0)}\int_{C_1 \leq \|\theta - \theta_0\| \leq C_2} \pi(\theta) d\theta \leq e^{-Nt - \frac{N}{2}\wh H(\theta_0)}\,.
    \end{multline*}  
    Moreover, according to the relation \eqref{eq:W_N}, $\frac{N}{2}(\partial_\theta \wh L'(\theta_0))^T (\partial_\theta^2 L(\theta_0)^{-1}) \partial_\theta \wh L'(\theta_0)$ converges in distribution to $Z^THZ$ with $Z \sim \mathcal{N}\l(0,I_d\r)$ where $I_d$ and $H$ are the $d$-dimensional identity matrix and $\frac{1}{2}I(\theta_0)(\partial_\theta^2L(\theta_0))^{-1}I(\theta_0)$ respectively. This implies that the sequence $e^{-\frac{N}{2}(\partial_\theta \wh L'(\theta_0))^T (\partial_\theta^2 L(\theta_0)^{-1}) \partial_\theta \wh L'(\theta_0)}$ is tight, and thus $O_p(1)$. Therefore,
\begin{multline*}
    O_p(N^{d/2}) e^{2C(\rho)\mathcal{O}\Delta_n\sqrt{n}} \int_{C_1 \leq \|\theta - \theta_0\| \leq C_2} \pi(\theta)e^{-N\l(\wh L'(\theta) - \wh L'(\theta_0)\r) - \frac{N}{2}\wh H(\theta_0)} d\theta \\
    \leq O_p(N^{d/2}) e^{2C(\rho)\mathcal{O}\Delta_n\sqrt{n} - Nt}\rightarrow 0
\end{multline*}
in probability. To this end, it suffices to show the convergence of 
\begin{align*}
    O_p(N^{d/2}) e^{2C(\rho)\mathcal{O}\Delta_n\sqrt{n}} \int_{\delta_N \leq \|\theta - \theta_0\| \leq C_1} \pi(\theta)e^{-N\l(\wh L'(\theta) - \wh L'(\theta_0)\r) - \frac{N}{2}\wh H(\theta_0)} d\theta
\end{align*} 
to $0$. Making the change of variables $h = (\theta - \theta_0)/\sqrt{N}$ and applying the inequality \eqref{eq:normal_plus_W_N}, we deduce that
\begin{multline*}
\int_{\delta_N \leq \|\theta - \theta_0\| \leq C_1} \pi(\theta)e^{-N\l(\wh L'(\theta) - \wh L'(\theta_0)\r) - \frac{N}{2}(\partial_\theta \wh L'(\theta_0))^T (\partial_\theta^2 L(\theta_0)^{-1}) \partial_\theta \wh L'(\theta_0)} d\theta \\
\leq C N^{-d/2}\int_{h: \sqrt{N}\delta_N \leq \|h\| \leq \sqrt{N}C_1} e^{-\l(h - \frac{1}{2} Z_N\r)^T\partial_\theta^2 L(\theta_0)\l(h - \frac{1}{2}Z_N\r) + \frac{N}{2}\wh H(\theta_0)} dh \,,
\end{multline*}
with probability approaching $1$ for sufficiently small $C_1$. Recall that $\frac{N}{2}\wh H(\theta_0) = O_p(1)$. Thus, the integral on the right hand side of the inequality can be bounded by $O_p(1) \mb P\l(\|U_N\| \geq \sqrt{N}\delta_N\mid  Z_N \r)$ where $U_N \sim \mathcal{N}\l(Z_N/2\,, \l(\partial_\theta^2 L(\theta_0)\r)^{-1}\r)$. Note that
\begin{multline*}
    \mb P\l(\|U_N\| \geq \sqrt{N}\delta_N\mid  Z_N \r) \\
    \leq \mb P\l(\|U_N - Z_N/2\| \geq \sqrt{N}\delta_N / 2 \mid Z_N\r)  + I\{\|Z_N / 2\| \geq \sqrt{N}\delta_N / 2\}\,.
\end{multline*}
Since $Z_N$ converges to a centered multivariate normal random variable in distribution by Lemma \ref{lemma:deriv-norm}, $I\{\|Z_N / 2\| \geq \sqrt{N}\delta_N / 2\}$ is $0$ with probability approaching $1$. Next, let $A$ be the Cholesky factor of $\l(\partial_\theta^2 L(\theta_0)\r)^{-1}$, that is, $A^TA = \l(\partial_\theta^2 L(\theta_0)\r)^{-1}$. Then
\begin{multline*}
    \mb P\l(\|U_N - Z_N / 2\| \geq \sqrt{N}\frac{\delta_N}{2}\r) = \mb P\l(\l\|A^T(A^T)^{-1}\l(U_N - Z_N / 2\r)\r\| \geq \sqrt{N}\frac{\delta_N}{2}\r) \\ \leq \mb P\l(\l\|(A^T)^{-1}\l(U_N - Z_N / 2\r)\r\| \geq \sqrt{N}\frac{\delta_N}{2\| A^T \|}\r) \\ = \mb P\l(\l\|(A^T)^{-1}\l(U_N - Z_N / 2\r)\r\|^2 \geq \frac{N\delta^2_N}{4\| A^T \|^2}\r)\,.
\end{multline*}
Here, $(A^T)^{-1}\l(U_N - Z_N / 2\r) \sim \mathcal{N}(0, I_d)$ and thus $\l\|(A^T)^{-1}\l(U_N - Z_N / 2\r)\r\|^2$ has $\chi^2$ distribution with $d$ degrees of freedom. By the tail probability estimate for $\chi^2$ distribution \cite[e.g., Lemma 1 in][]{laurent2000adaptive}, for $N$ large enough,
\begin{align*}
    \mb P\l(\l\|(A^T)^{-1}\l(W_N - Z_N / 2\r)\r\|^2 \geq \frac{N\delta^2_N}{4\| A^T \|^2}\r) \leq e^{-C_d\frac{N\delta_N^2}{\|A^T\|^2}}\,.
\end{align*}
Therefore, 
\begin{multline*}
    O_p(N^{d/2}) e^{2C(\rho)\mathcal{O}\Delta_n\sqrt{n}} \times \\ \int_{\delta_N/2 \leq \|\theta - \theta_0\| \leq C_1} \pi(\theta)e^{-N\l(\wh L'(\theta) - \wh L'(\theta_0)\r) - \frac{N}{2}\wh H(\theta_0)} d\theta \\
    \leq O_p(1)e^{2C(\rho)\mathcal{O}\Delta_n\sqrt{n} - C_d N \delta_N^2}\,.
\end{multline*}
According to the assumption of the theorem, the quantity above converges to $0$ in probability as $N\to\infty$.

\end{document}